\documentclass[a4paper,10pt]{amsart}
\usepackage{amsfonts}
\usepackage{xypic}
\usepackage{amsthm}
\usepackage{amssymb}
\usepackage{amsmath}
\usepackage{enumerate}
\usepackage{graphicx}
\usepackage{array}

\begin{document}
\title{Large values of Hecke-Maass $L$-functions with prescribed argument}
\author{Alexandre Peyrot}
\address{EPFL\\ SB MATHGEOM TAN\\ MA C3 614, Station 8\\ CH-1015 Lausanne\\ Switzerland}
\email{alexandre.peyrot@epfl.ch}

\maketitle

\def \l {{\lambda}}
\def \a {{\alpha}}
\def \b {{\beta}}
\def \f {{\phi}}
\def \r {{\rho}}
\def \R {{\mathbb R}}
\def \H {{\mathbb H}}
\def \N {{\mathbb N}}
\def \C {{\mathbb C}}
\def \Z {{\mathbb Z}}
\def \F {{\Phi}}
\def \Q {{\mathbb Q}}
\def \e {{\epsilon }}
\def\GL{\ensuremath{\mathop{\textrm{\normalfont GL}}}}
\def\SL{\ensuremath{\mathop{\textrm{\normalfont SL}}}}
\def\Gal{\ensuremath{\mathop{\textrm{\normalfont Gal}}}}
\def\SU{\ensuremath{\mathop{\textrm{\normalfont SU}}}}
\def\SO{\ensuremath{\mathop{\textrm{\normalfont SO}}}}

\newtheorem{prop}{Proposition}
\newtheorem{claim}{Claim}
\newtheorem{lemma}{Lemma}
\newtheorem{thm}{Theorem}
\newtheorem{ktf}{\textit{Kuznetsov Trace Formula}}
\newtheorem{defn}{Definition}

\theoremstyle{definition}
\newtheorem{exmp}{Example}

\theoremstyle{remark}
\newtheorem{rmk}{Remark}

\newcommand{\mods}[1]{\,(\mathrm{mod}\,{#1})}

\begin{abstract}
We investigate the existence of large values of $L$-functions attached to Maass forms on the critical line with prescribed argument. The results obtained rely on the resonance method developed by Soundararajan and furthered by Hough.
\end{abstract}

\setcounter{tocdepth}{1}
\tableofcontents

\section{Introduction and Setup}
The resonance method developed by Soundararajan \cite{Soundres} allows the detection of large values of certain $L$-functions on the critical line. Building on this work, Hough \cite{Hough} proves the existence of large values of the Riemann zeta function on the critical line with prescribed argument. In this paper we extend the resonance method to find large values of Hecke-Maass $L$-functions on the critical line with prescribed argument. More precisely, we let $f$ be an (even) Hecke-Maass eigenform for $\SL_2(\Z)$, and denote by $1/4 + r^2$ the associated eigenvalue of the Laplacian. We define the Hecke operators $(T_n)_{n\geq 1}$ acting on the space of Maass forms by 
$$(T_n f)(z) = \frac{1}{\sqrt{n}} \sum_{ad=n} \sum_{0\leq b <d} f\left(\frac{az+b}{d}\right).$$
We associate to $f$ the sequence of Hecke-eigenvalues $(\lambda_f(n))_{n\geq 1}$. We define the associated $L$-function,
$$L(f,s):= \sum_n \frac{\l_f(n)}{n^s}= \prod_p (1-\alpha_pp^{-s})^{-1}(1-\beta_pp^{-s})^{-1},$$
where $\alpha_p, \beta_p$ are given via $\alpha_p+\beta_p = \lambda_f(p)$ and $\alpha_p\beta_p=1$. We prove the following theorem.
\begin{thm}\label{peythm}
For any $\eta <1$, any sufficiently large $T\in \R$ and any $\theta\in \R/\Z$, there exists $t\in[\frac{T}{2},2T]$ such that
$$\frac{1}{2\pi} \arg L\left(f,\frac{1}{2}+it\right) \equiv \theta \mod \Z, \hbox{ and } \log\left|L\left(f,\frac{1}{2}+it\right)\right| \geq (\eta+o(1)) \sqrt{\frac{\log T}{\log\log T}}.$$
\end{thm} 
We follow Hough's strategy \cite{Hough}, namely we exploit sign changes of $L(f,s)$ by comparing the weighted signed moment and unsigned first moment, which we define in the next section. Several substantial complications, however, arise due to the fact that $L(f,s)$ is of degree 2. We may no longer exploit combinatorial arguments to handle sums of fractional divisor functions. We treat these sums by relating them to the symmetric square $L$-function, $L(\hbox{sym}^2 f, s)$, and exploiting a zero-free region. 

We note that the results presented also hold for holomorphic cusp forms, as they exhibit the same properties as those exploited for Maass forms. Moreover, we expect that the methods are flexible enough to carry over to the case of  Maass forms of $\SL_n(\Z)$ \footnote{Either for self-dual forms, or in the case that the form satisfies Ramanujan-Petersson by a recent non-zero region due to Goldfeld and Li \cite{goldfeldli}.} , by some more elaborate calculations.

\subsection{Outline of proof}
Following \cite{Hough}, we implement the resonance method developed in \cite{Soundres}. We thus let $T$ be a large real number and $\theta \in \R$ be a fixed angle. Let $\xi>0$ be a  small real number and let $N= T^{1-3\xi}$. We set $L= \sqrt{\log N\log\log N}$, and define the multiplicative function, $r(n)$, which is supported on square-free integers and defined at primes by
$$r(p) = \left\{\begin{array}{ll} \frac{L}{\sqrt{p}\log p}, & \hbox{ if } L^2\leq p\leq \exp((\log L)^2)\\ 0, & \hbox{ otherwise } \end{array}\right. .$$
We define a preliminary resonating polynomial,
$$R^*(s) = \sum_{n\leq N} \frac{r(n)\l_f(n)}{n^s}.$$
We also introduce a short Dirichlet polynomial, 
$$A_{1/2}(s):= \sum_{n\leq T^\xi} \frac{d_{1/2}(n) \l_f(n)}{n^s},$$
where $d_{1/2}$ are the Dirichlet series coefficients for $\zeta^{1/2}$. In particular $d_{1/2}$ is multiplicative, non-negative, and is given at prime powers by
\begin{equation*}
d_{1/2}(p^k) =  \frac{1}{2^k k!} \prod_{i=1}^k (2i-1).
\end{equation*}
We define our resonating polynomial to be
$$R(s)= R^*(s) A_{1/2}\left(\frac{1}{2}+s\right) =: \sum_{n\leq T^{1-2\xi}} \frac{a_n}{n^s}.$$
In order to prove Theorem \ref{peythm} we compute weighted first moments of $L(f,\frac{1}{2} +it)$. Namely, we let
$$T_\theta := \left\{t\in\R|\arg\left(L\left(f,\frac{1}{2}+it\right)\right) \equiv \theta \mods \pi \right\},$$
and letting $H= T/(\log T)^2$, we define
$$\omega_{T,\theta}(t) = \frac{|R(it)|^2}{\cosh\left(\frac{t-T}{H}\right)}\bigg/NW,$$
where
\begin{equation}\label{normweight}
NW:= \sum_{t \in T_\theta} \frac{|R(it)|^2}{\cosh\left(\frac{t-T}{H}\right)},
\end{equation}
 is the normalizing weight required to obtain a probability measure. Theorem \ref{peythm} will be deduced from the following proposition.
\begin{prop}\label{prop1}
We have 
\begin{equation}\label{unsignedmoment}
\sum_{t\in T_\theta} \left|L\left(f,\frac{1}{2}+it\right)\right| \omega_{T,\theta}(t) \gg (\log T)^{\frac{3}{4}} \prod_p \left(1+ \frac{r(p)\l_f^2(p)}{\sqrt{p}}\right),
\end{equation}
and
\begin{equation}\label{signedmoment}
\sum_{t\in T_\theta} L\left(f,\frac{1}{2}+it\right) \omega_{T,\theta}(t) \ll (\log T)^{\frac{1}{2}} \prod_p \left(1+ \frac{r(p)\l_f^2(p)}{\sqrt{p}}\right).
\end{equation}
\end{prop} 
We explain here the strategy that allows us to detect the angle of $L(f,s)$ thus allowing us to estimate these moments. Let
$$\Lambda(f,s) = L_\infty(s) L(f,s),$$
be the completed $L$-function of $f$, where
$$L_\infty(s):= \pi^{-s} \Gamma\left(\frac{s+ir}{2}\right) \Gamma\left(\frac{s-ir}{2}\right),$$
is the local factor at $\infty.$ The $L$-function satisfies the functional equation:
$$\Lambda(f,s)=\Lambda(f,1-s).$$
We let 
$$\Delta(s):= \frac{L\left(f,\frac{1}{2}+s\right)}{L\left(f,\frac{1}{2}-s\right)} =\frac{L_\infty\left(\frac{1}{2}-s\right)}{L_\infty\left(\frac{1}{2}+s\right)},$$
and observe that the points, $t$, such that  $\arg(L(f,\frac{1}{2}+it))=\theta \mods \pi$ are the solution set of $\Delta(it) =e^{2i\theta}.$ In particular, we note that $T_\theta$ is not empty. By the Residue Theorem, one may then express the moment as a contour integral of the form 
$$\int_\Gamma L\left(f,\frac{1}{2}+s\right) R(s)R(-s) \frac{\Delta'(s)}{\Delta(s)-e^{2i\theta}} \frac{\mathrm{d}s}{\cos\left(\frac{s-iT}{H}\right)},$$
where $\Gamma$ is an appropriate contour supported at height $T$. Expanding the $L$-function into its Dirichlet series we end up having to estimate sums of Hecke eigenvalues against certain arithmetic functions.

We end this section by showing how Theorem \ref{peythm} follows from Proposition \ref{prop1}. By Proposition \ref{prop1}, we have 
\begin{small}
\begin{align*}
\sum_{ \arg(L)= \theta} \left|L\left(f,\frac{1}{2}+it\right)\right| \omega_{T,\theta}(t) &= \frac{1}{2} \sum_{t\in T_\theta} \left(\left|L\left(f,\frac{1}{2}+it\right)\right| + e^{-i\theta} L\left(f,\frac{1}{2}+it\right)\right)\omega_{t,\theta}(t)\\
& \gg (\log T)^{3/4} \prod_p\left(1+\frac{r(p)\l_f(p)^2}{\sqrt{p}}\right),
\end{align*}
\end{small}
so that 
$$\max_{\substack{\frac{T}{2}\leq t \leq 2T\\ \arg(L)=\theta}} \left|L\left(f,\frac{1}{2}+it\right)\right| \gg  (\log T)^{3/4} \prod_p\left(1+\frac{r(p)\l_f(p)^2}{\sqrt{p}}\right).$$
Theorem \ref{peythm} now follows from 
\begin{align*}
\log \prod_p \left(1+\frac{r(p)\r_f(p)^2}{\sqrt{p}}\right)  & \sim L \sum_{L^2\leq p \leq \exp(\log^2 L)} \frac{\l_f(p)^2}{p \log p} \\
& \sim \sqrt{(1-3\xi)\frac{\log T}{\log\log T}},
\end{align*}
and letting $\xi \rightarrow 0$.

\subsection{Notations}
Throughout the paper, we will let $f(x) \ll g(x),$ $f(x) \gg g(x)$ and $f(x) = O(g(x))$ denote the usual Vinogradov symbols. The notation $f(x) \asymp g(x)$ will be used to mean that both $f(x) \ll g(x)$ and $g(x) \ll f(x)$ hold. The notation $f(x) \sim g(x)$ will be taken to mean that $\lim_{x\rightarrow \infty} f(x)/g(x) =1$. We will write $f(x) = o(g(x))$ to mean that $ \lim_{x\rightarrow \infty} f(x)/g(x) \rightarrow 0$. We also follow the convention that any $\epsilon$ appearing in the paper is defined to be an arbitrarily small unspecified positive real number, that might vary from one line to the other. Whenever we encounter a zero-free region for an $L$-function, we will take the $k$-th root of $L$ in that region to be the one defined so that $L^{1/k}\rightarrow 1$ as $L\rightarrow 1$ with $s \rightarrow \infty, s\in \R$.

\subsection{Acknowledgements}
I am grateful toward Philippe Michel for the guidance and support received throughout this project. I would also like to thank Pierre Le Boudec, Paul Nelson, Ramon M. Nunes, Yiannis Petridis  and Ian Petrow for useful comments and suggestions.

\section{Preliminary lemmas}
In order to estimate these moments, we will require some preliminary lemmas that we prove in this section. 

\begin{lemma}\label{prelimlemma}
Let $T$ be large, and $1\leq m,n$ and assume $m< T^{2-\delta}$, and $\min(m,n) < T^{1-\delta}$ for some $\delta >0$. We then have for any $\omega \in \mathbb{S}^1$ and for any $A>0$,
\begin{equation}\label{third}
\int_{\Re(s)=\frac{1}{2}+\epsilon} \left(\frac{m}{n}\right)^s \frac{\Delta'(s)}{\Delta(s)} \frac{\Delta(s)}{1-\omega\Delta(s)} \frac{\mathrm{d}s}{\cos\left(\frac{iT\pm s}{H}\right)} = O_{\delta,A}(T^{-A}).
\end{equation}
Letting 
$$I_T:= \int_{t\geq 20} \frac{-2 \Delta'(it)/\Delta(it)}{\cosh\left(\frac{t-T}{H}\right)} \mathrm{d}t,$$
we also have,
\begin{align}\label{second}
&\frac{1}{2\pi i} \int_{\Re(s)=\frac{1}{2}+\epsilon} \left(\frac{m}{n}\right)^s \frac{\Delta'(s)}{\Delta(s)} \frac{\mathrm{d}s}{\cos\left(\frac{iT\pm s}{H}\right)}  = -\frac{\delta_{m=n}}{4\pi} I_T + O_{\delta,A}(T^{-A}).
\end{align}
\end{lemma}

\begin{proof}
We need some estimates about $\Delta(s)$. By Stirling's formula, we have that for $|t|\gg1$,
\begin{align*}
|\Delta(\sigma+it)|&\ll \left|\frac{\pi^{2(\sigma+it)}\Gamma\left(\frac{\frac{1}{2}-\sigma+i(r-t)}{2}\right) \Gamma\left(\frac{\frac{1}{2}-\sigma-i(r+t)}{2}\right)}{\Gamma\left(\frac{\frac{1}{2}+\sigma+i(r+t)}{2}\right) \Gamma\left(\frac{\frac{1}{2}+\sigma+i(t-r)}{2}\right)}\right| \ll_\sigma t^{-2\sigma}.
\end{align*}
Writing $\Delta(s)=\pi^{2s}\Gamma_1\Gamma_2/(\Gamma_3\Gamma_4)$, we compute
\begin{align}\label{logder}
\nonumber \frac{\Delta'(it)}{\Delta(it)}&= 2\log\pi + \frac{\Gamma_1'}{\Gamma_1} (it)+ \frac{\Gamma_2'}{\Gamma_2}(it) - \frac{\Gamma_3'}{\Gamma_3} (it)- \frac{\Gamma_4'}{\Gamma_4}(it) \\
&= -\frac{1}{2} \log\left(\frac{\frac{1}{16}+\frac{1}{4}((r+t)^2+(t-r)^2)+(t^2-r^2)^2}{16\pi^4}\right) + O(|t|^{-1+\epsilon}),
\end{align}
and thus also 
$$\frac{\mathrm{d}^j}{\mathrm{d}t^j} \frac{\Delta'}{\Delta}(it) = O_j(|t|^{1-j}),$$
for $j\geq 2$. In order to prove (\ref{third}), we push the line of integration rightwards to $\Re(s)=(A+1)/\delta + \delta'$, with $0<\delta'<1$ chosen so that the contour has a distance bounded from any pole of the integrand. 
In pushing the line rightwards as indicated above, the only poles we pass are counter-weighted by the hyperbolic cosine factor (since these poles can only occur for $t$ bounded away from the real axis) and they therefore contribute a negligible amount. We are thus left with estimating 
\begin{align*}
&\int_{\Re(s)=(A+1)/\delta + \delta'} \left(\frac{m}{n}\right)^s \frac{\Delta'}{\Delta}(s) \frac{\Delta(s)}{1-\omega\Delta(s)} \frac{\mathrm{d}s}{\cos\left(\frac{iT\pm s}{H}\right)} \\
&\hspace{5 mm}\ll \int_{\frac{T}{2}}^{2T} T^{(2-\delta) \left(\frac{A+1}{\delta}+\delta'\right)} \log(|t|) T^{-2\left(\frac{A+1}{\delta}+\delta'\right)} \mathrm{d}t +O(|T|^{-A})\\
&\hspace{5 mm} \ll T^{-A}.
\end{align*}
In order to prove (\ref{second}), we note that $\Delta(s)$ has no poles nor zeroes on $\Re(s)=0$, and as before the only poles we might encounter are negligible, and we may thus shift our line of integration to $\Re(s)=0$. By (\ref{logder}), the integral becomes 
\begin{align*}
-\frac{1}{4\pi} \int_\R \left(\frac{m}{n}\right)^{it}\frac{\log\left(\frac{\frac{1}{16}+\frac{1}{4}((r+t)^2+(t-r)^2)+(t^2-r^2)^2}{16\pi^4}\right)+O(|t|^{-1+\epsilon})}{\cosh\left(\frac{T\pm t}{H}\right)} \mathrm{d}t.
\end{align*}
If $m\not= n$, then by repeated integration by parts, the integral is negligible. The lemma follows. 
\end{proof}
We note that $I_T$ satisfies
$$I_T= \int_{t\geq 20} \frac{\log\left(\frac{\frac{1}{16}+\frac{1}{4}((r+t)^2+(t-r)^2)+(t^2-r^2)^2}{16\pi^4}\right)+O(|t|^{-1})}{\cosh\left(\frac{t-T}{H}\right)} \mathrm{d}t.$$

We recall that by the analog of Mertens' Theorem for Rankin-Selberg $L$-functions, there exists a constant, $C$, such that 
$$\sum_{p\leq x} \frac{\l_f(p)^2}{p} = \log \log x + C + o(1),$$
and will use it without mention in the proof of the following lemmas.
\begin{lemma}\label{lemma2}
For any $|\alpha| \leq \frac{1}{(\log L)^3}$, we have 
\begin{align*}
&\log \prod_p (1+r(p)^2\l_f(p)^2p^\alpha) - \log \prod_p (1+ r(p)^2\l_f(p)^2) \\
& \hspace{5 mm} \leq  \alpha \left(\log N- (1+o(1)) \frac{\log N \log\log \log N}{\log \log N}\right).
\end{align*}
\end{lemma}

\begin{proof}
We write 

\begin{align*}
&\log \prod_p (1 + r(p)^2 \l_f(p)^2 p^\alpha) - \log \prod_p (1+ r(p)^2 \l_f(p)^2) \\
&\hspace{5 mm}= \sum_{L^2 \leq p \leq \exp(\log^2 L)} \log\left(1 + \frac{r(p)^2 \l_f(p)^2 (p^\alpha -1)}{1+ r(p)^2 \l_f(p)^2}\right)\\
&\hspace{5 mm}= \sum_{L^2 \leq p \leq \exp(\log^2 L)} \frac{r(p)^2 \l_f(p)^2 (p^\alpha -1) }{ 1 +r(p)^2\l_f(p)^2} \left(1+O\left( \frac{r(p)^2 \l_f(p)^2 (p^\alpha-1)}{1 + r(p)^2 \l_f(p)^2}\right)\right).   
\end{align*}
Since, 
$$(p^\alpha-1)\frac{r(p)^2 \l_f(p)^2}{1+r(p)^2 \l_f(p)^2} \leq p^\alpha -1 \ll \alpha \log p \ll \frac{1}{\log L},$$ 
we may bound the difference of logarithms by

\begin{align*}
& \sum_{L^2 \leq p \leq \exp(\log^2 L)} r(p)^2 \l_f(p)^2 (p^\alpha -1 ) \left( 1+ O \left( \frac{1}{\log L}\right) \right) \\
&\hspace{5 mm}= \alpha L^2 \sum_{L^2 \leq p \leq \exp(\log^2 L)} \frac{\l_f(p)^2}{p} \frac{1}{\log p} \left(1 + O\left(\frac{1}{\log L}\right)\right)\\
&\hspace{5 mm}= \alpha L^2 \left( \frac{\log(\log^2 L) + C + o(1)}{\log^2 L} - \frac{\log\log L^2 + C + o(1)}{\log L^2} \right.\\
&\left. \hspace{2 cm}+ \int_{L^2}^{\exp(\log^2 L)} \frac{\log\log x + C + o(1)}{(\log x)^2 x} \mathrm{d}x \right)  \left(1+O\left(\frac{1}{\log L}\right)\right)\\
&\hspace{5 mm}= \alpha L^2 \left( \frac{1}{2\log L}  + o \left( \frac{1}{\log L}\right)\right)  \left(1+O\left(\frac{1}{\log L}\right)\right)\\
&\hspace{5 mm} = \alpha \left( \log N -  (1+o(1)) \frac{\log N \log\log\log N}{\log\log N}\right).
\end{align*}
\end{proof}

As a corollary, we deduce the following lemma.
\begin{lemma}\label{r2g2}
For any integer $l\geq 1$ and for any $Z> N\exp\left(-\frac{\log N}{(\log\log N)^2}\right),$
$$\sum_{\substack{n< Z \\ (n,l)=1}} r(n)^2 \l_f(n)^2 = \left(1+O\left(\exp\left(-\frac{L^2}{(\log L)^5}\right)\right)\right) \prod_{p \nmid l} (1+r(p)^2 \l_f(p)^2).$$
\end{lemma}

\begin{proof}
We use Rankin's trick to write
\begin{align*}
\sum_{\substack{n<Z \\ (n,l)=1}} r(n)^2\l_f(n)^2 &= \sum_{\substack{n=1 \\ (n,l)=1}}^\infty r(n)^2\l_f(n)^2 - \sum_{\substack{n \geq Z \\ (n,l)=1}} r(n)^2\l_f(n)^2\\
& = \prod_{p \nmid l} (1+ r(p)^2\l_f(p)^2) + O(Z^{-\alpha} \prod_{p\nmid l} (1+p^\alpha r(p)^2 \l_f(p)^2)).
\end{align*}
The result then follows immediately from Lemma \ref{lemma2}.
\end{proof}
We now prove analogously the following two Lemmas.
\begin{lemma}\label{lemma4}
For any $|\alpha| \leq \frac{1}{(\log L)^3}$, and any multiplicative function, $g$, such that for some $m>0$, $0\leq g(p) \leq m$ for all $p$, we have
\begin{align*}
\log\left(\prod_p \frac{1+r(p)\l_f(p)^2g(p)p^{\alpha-1/2}}{1+ r(p)\l_f(p)^2g(p)p^{-1/2}}\right) \ll_m \alpha L \log\log L .
\end{align*}
\end{lemma}

\begin{proof}
We may write
\begin{small}
\begin{align*}
 &\sum_{L^2\leq p \leq \exp(\log^2 L)} \log\left(1+ \frac{r(p)\l_f(p)^2g(p)p^{-1/2}(p^\alpha-1)}{1+r(p)g(p)\l_f(p)^2 p^{-1/2}}\right)\\
&\hspace{5 mm}= \sum_{L^2\leq p\leq \exp(\log^2 L)} \frac{r(p)\l_f(p)^2 g(p)p^{-1/2}(p^\alpha-1)}{1+r(p)g(p)\l_f(p)^2p^{-1/2}} \left(1+O\left( \frac{r(p)\l_f(p)^2g(p)p^{1/2}(p^\alpha-1)}{1+r(p)\l_f(p)^2 g(p)p^{-1/2}}\right)\right).
\end{align*}
\end{small}
Since 
$$(p^\alpha-1)\frac{r(p)\l_f(p)^2g(p)p^{-1/2}}{1+r(p)g(p)\l_f(p)^2p^{-1/2}} \leq p^\alpha-1\ll \alpha\log p\ll \frac{1}{\log L},$$
we bound
\begin{align*}
& \log\left(\prod_p \frac{1+r(p)\l_f(p)^2g(p) p^{\alpha-1/2}}{1+r(p)\l_f(p)^2g(p)p^{-1/2}}\right)\\
&\hspace{5 mm} \leq \sum_{L^2\leq p \leq \exp(\log^2 L)} r(p) \l_f(p)^2 g(p) p^{-1/2} (p^\alpha-1) \left(1+ O\left(\frac{1}{\log L}\right)\right)\\
&\hspace{5 mm} \ll_m \sum_{L^2 \leq p \leq \exp(\log^2L)} \alpha \frac{L}{p}\l_f(p)^2 \left(1+O\left(\frac{1}{\log L}\right)\right)\\
& \hspace{5 mm}= \alpha L  (\log(\log^2 L)  - \log\log L^2 +o(1)) \left(1+O\left(\frac{1}{\log L}\right)\right)\\
&\hspace{5 mm} \ll \alpha L \log\log L.
\end{align*}
\end{proof}

As a corollary we deduce the following Lemma.

\begin{lemma}\label{rr2g}
For $Z> \exp(L(\log L)^5)$, and $g$ multiplicative such that for some m, $0\leq g(p) \leq m$ for all $p$, we have
$$\sum_{n\geq Z} \frac{r(n)}{\sqrt{n}} \l_f(n)^2 g(n) \leq \exp\left(-(1+o_m(1)) \frac{\log Z}{(\log L)^3}\right),$$
and
$$\sum_{n<Z} \frac{r(n)}{\sqrt{n}} \l_f(n)^2g(n) = \left(1+O_m\left(\exp(-cL(\log L)^2)\right)\right) \prod_p\left(1+\frac{r(p)}{\sqrt{p}}\l_f(p)^2g(p)\right).$$
\end{lemma}

\begin{proof}
We use Rankin's trick to write 
\begin{align*}
\sum_{n<Z} \frac{r(n)}{\sqrt{n}} \l_f(n)^2g(n) &= \sum_{n=1}^\infty \frac{r(n)}{\sqrt{n}} \l_f(n)^2 g(n) - \sum_{n\geq Z} \frac{r(n)}{\sqrt{n}} \l_f(n)^2 g(n)\\
&= \prod_p \left(1+ \frac{r(p)}{\sqrt{p}} \l_f(p)^2 g(p)\right)\\
& + O\left(Z^{-\alpha} \prod_p (1+r(p)\l_f(p)^2g(p)p^{\alpha-1/2})\right). 
\end{align*}
The result now follows from Lemma \ref{lemma4}.
\end{proof}

Throughout the paper, we will also require a result of Tenenbaum \cite[Theorem 5.2, p. 281]{Tenenbaumbook}, inspired by previous work of Delange \cite{Delange54, Delange71}, that we give in the following lemma. We first need to set up some notation. Let $z\in \C$, and fix $c_0>0, 0<\delta\leq 1, M>0$, positive constants. Writing $s=\sigma +i\tau$, we say that a Dirichlet series $F(s)$ has the property $\mathcal{P}(z;c_0,\delta,M)$ if the Dirichlet series 
$$G(s;z):= F(s)\zeta(s)^{-z}$$
may be continued as a holomorphic function for $\sigma \geq 1- c_0/(1+ \log (2+|\tau|))$, and, in this domain, satisfies the bound 
$$|G(s;z)|\leq M (1+|\tau|)^{1-\delta}.$$
If $F(s)= \sum a_n/n^s$ has the property $\mathcal{P}(z;c_0,\delta,M)$, and if there exists a sequence of non-negative real numbers $\{b_n\}_{n=1}^\infty$ such that $|a_n| \leq b_n$, $(n=1,2,\cdots)$, and the series 
$$\sum_{n\geq 1} \frac{b_n}{n^s}$$
satisfies $\mathcal{P}(w;c_0,\delta,M)$ for some complex number $w$, we shall say that $F(s)$ has type $\mathcal{T}(z,w;c_0,\delta,M)$.  

\begin{lemma}\label{tenen}
Let $F(s):=\sum a_n/n^s$ be a Dirichlet series of type $\mathcal{T}(z,w;c_0,\delta,M)$. For $x\geq 3, A>0, |z|\leq A, |w|\leq A$, there exist $d> 0$ such that
$$\sum_{n\leq x} a_n = x(\log x)^{z-1} \left\{\frac{G(1;z)}{\Gamma(z)} + O( M(e^{-d \sqrt{\log x}} + \log x^{-1}))\right\}.$$
The constant $d$ and the implicit constant in the Landau symbol depend at most on $c_0, \delta,$ and $A$.
\end{lemma}

\section{Computing the normalizing weight}
In this section we compute the normalizing weight, $NW$, given by (\ref{normweight}). We will require the following estimates on the coefficients $a_n$.

\begin{lemma}\label{an2}
We have
\begin{equation}\label{l7}
\sum_{n\leq T^{1-2\xi}} a_n^2 \asymp (\log T)^{1/4} \prod_p (1+r(p)^2\l_f(p)^2) \prod_p\left(1+\frac{r(p)\l_f(p)^2}{\sqrt{p}} \right).
\end{equation}
\end{lemma}

\begin{proof}
We have
$$a_n= \sum_{\substack{l\leq T^{1-3\xi}\\ m\leq T^\xi \\ lm=n}} r(l)\l_f(l)\frac{d_{1/2}(m)\l_f(m)}{m^{1/2}},$$
so that 
\begin{small}
\begin{align*}
\sum_{n \leq T^{1-2\xi}} a_n^2 &= \sum_{l_1,l_2\leq T^{1-3\xi}} r(l_1)r(l_2) \l_f(l_1)\l_f(l_2) \sum_{\substack{n_1,n_2\leq T^\xi\\ l_1n_1=l_2n_2}} \frac{d_{1/2}(n_1)d_{1/2}(n_2) \l_f(n_1)\l_f(n_2)}{(n_1n_2)^{1/2}}\\
&= \sum_{g\leq T^{1-3\xi}} r(g)^2 \l_f(g)^2 \sum_{\substack{l_1,l_2 \leq T^{1-3\xi}/g \\ (l_1,l_2)=(l_1l_2,g)=1}} r(l_1l_2) \l_f(l_1l_2)\\
&\hspace{1 cm}\times  \sum_{\substack{n_1,n_2 \leq T^\xi \\ l_1n_1 = l_2n_2}} \frac{d_{1/2}(n_1)d_{1/2}(n_2) \l_f(n_1)\l_f(n_2)}{(n_1n_2)^{1/2}} .
\end{align*}
\end{small}
We now let $n_2 : = n_2/l_1$ and $n_1:=n_1/l_2$, so that we may rewrite this as 
\begin{align*}
 & \sum_{g\leq T^{1-3\xi}} r(g)^2\l_f(g)^2 \sum_{\substack{l_1,l_2 \leq T^{1-3\xi}/g\\ (l_1,l_2)=(l_1l_2,g)=1}} \frac{r(l_1l_2)}{(l_1l_2)^{1/2}} \l_f(l_1l_2)\\
 & \sum_{n\leq T^\xi/\max(l_1,l_2)} \frac{d_{1/2}(l_1n)d_{1/2}(l_2n)\l_f(l_1n)\l_f(l_2n)}{n}.
\end{align*}
\subsection{The $n$-sum} The idea is to treat the innermost sum by relating it to the fourth root of the Rankin-Selberg $L$-function,
\begin{align*}
L(f\times f,s) &:= \prod_p (1-p^{-s})^{-1} (1-\alpha_p^2p^{-s})^{-1}(1-\beta_p^{2}p^{-s})^{-1}(1- \alpha_p\beta_pp^{-s})^{-1}\\
&= \zeta(s) L(\hbox{sym}^2 f,s),
\end{align*}
where 
$$L(\hbox{sym}^2f,s) := \prod_p (1-p^{-s})^{-2} (1-\alpha_p^2p^{-s})^{-1}(1-\beta_p^{2}p^{-s})^{-1},$$
denotes the symmetric square $L$-function as studied by Gelbart and Jacquet \cite{GelbartJacquet}. 
Following \cite[Chapter II.5]{Tenenbaumbook}, we define the generalized binomial coefficient by
\begin{align*}
&\left(\begin{array}{l} \omega \\ \nu \end{array} \right) := \frac{1}{\nu!} \prod_{0\leq j < \nu}(\omega - j) & (\omega \in \C, \nu\in \N),
\end{align*}
so that
\begin{align*}
L^{1/4}(f\times f,s) =& \prod_p (1-p^{-s})^{-1/2} (1-\alpha_p^2p^{-s})^{-1/4}(1-\beta_p^{2}p^{-s})^{-1/4}\\
=& \prod_p \left(\sum_{k=0}^\infty \left(\begin{array}{c} k-\frac{1}{2} \\ k \end{array} \right) p^{-ks}\right) \left(\sum_{k=0}^\infty \left( \begin{array}{c} k-\frac{3}{4} \\ k \end{array}\right) \alpha_p^{2k}p^{-ks}\right) \\
& \times \left( \sum_{k=0}^\infty \left(\begin{array}{c} k-\frac{3}{4} \\k \end{array}\right) \beta_p^{2k}p^{-ks}\right)  \\
=& \prod_p \left(\sum_{k=0}^\infty a(p^k)p^{-ks}\right), 
\end{align*}
where $a$ is a multiplicative function such that 
$$a(p) = \frac{\l_f(p)^2}{4} = d_{1/2}^2(p) \l_f^2(p).$$
Given that $L(\hbox{sym}^2f,s)$ is a cuspidal automorphic $L$-function (see \cite{GelbartJacquet}), writing $s= \sigma + i \tau$,   there exists a constant $c>0$, depending on $f$, such that  $\zeta(s)L(\hbox{sym}^2 f,s)$ is non-zero in the region $\sigma > 1- c/\log(2+|\tau|)$ (see \cite{MichelParkcity}). We note that $L(\hbox{sym}^2 f,s)$ is entire in that region, so that 
$$\sum_{n=1}^\infty \frac{d_{1/2}^2(n) \l_f^2(n)}{n^s} = \zeta^{1/4}(s) L^{1/4}( \hbox{sym}^2 f,s) F(s),$$
where $F(s)$ is a non-zero, bounded and holomorphic function in the region $\sigma > 1- c/\log(2+|\tau|)$. It follows that
\begin{align*}
&\sum_{n=1}^\infty \frac{d_{1/2}(l_1n) d_{1/2}(l_2n) \l_f(l_1n)\l_f(l_2n)}{n^{s}} \\
&= \prod_{p | l_1l_2} \left( \sum_{k=0}^\infty \frac{d_{1/2}(p^{k+1}) d_{1/2}(p^k) \l_f(p^{k+1}) \l_f(p^k)}{p^{ks}}\right) \times \prod_{p\nmid l_1l_2} \left(\sum_{k=0}^\infty \frac{d_{1/2}^2(p^k)\l_f^2(p^k)}{p^{ks}}\right)\\
&= G(s;l_1l_2) \prod_p \left(\sum_{k=0}^\infty \frac{d_{1/2}^2(p^k)\l_f^2(p^k)}{p^{ks}}\right) \\
&= G(s;l_1l_2) F(s) L^{1/4}(\hbox{sym}^2 f,s) \zeta^{1/4}(s),
\end{align*}
where
$$G(s;l) = \prod_{p | l} \frac{\sum_{k=0}^\infty \frac{d_{1/2}(p^{k+1})d_{1/2}(p^k) \l_f(p^{k+1}) \l_f(p^k)}{p^{ks}}}{\sum_{k=0}^\infty \frac{d_{1/2}^2(p^k) \l_f^2(p^k)}{p^{ks}}}.$$
Observe that the denominator is non-zero because the coefficients are positive. We let 
$$G(s;1/4,l_1,l_2):= L^{1/4}(\hbox{sym}^2 f,s) F(s) G(s,l_1,l_2),$$
and wish to bound $|G(s;1/4,l_1,l_2)|$ in the aforementioned domain. Noting that $k$ and $k+1$ have distinct parity, we estimate
\begin{align*}
|G(s,l_1l_2)|  &= \left|\prod_{p | l_1l_2} \frac{\sum_{k=0}^\infty \frac{d_{1/2}(p^{k+1})d_{1/2}(p^k)\l_f(p^{k+1})\l_f(p^k)}{p^{ks}}}{\sum_{k=0}^\infty \frac{d_{1/2}^2(p^k)\l_f^2(p^{k})}{p^{ks}}}\right|\\
&\leq d_{1/2}(l_1l_2)|\l_f(l_1l_2) \prod_{p|l_1l_2} M_G(l_1l_2)|,
\end{align*}   
where $M_G$ is a multiplicative function supported on squarefree integers satisfying at primes
\begin{equation}\label{mp}
M_G(p)^{\pm 1} \leq (1 + C_1 p^{-\delta_1}),
\end{equation}
for some $\delta_1>0$ and an absolute constant $C_1$ (one may use bounds towards the Ramanujan-Petersson conjecture as given in \cite{KimSarnak}). Since $L^{1/4}(s,\hbox{sym}^2 f) \ll \tau^{\delta}$ for any arbitrarily small $\delta>0$, and letting $M >0$ be such that $F(s)\leq M$ in that region, we conclude that
$$|G(s;1/4,l_1,l_2)| \leq M | \l_f(l_1l_2)|d_{1/2}(l_1l_2) M_G(l_1l_2) (1+|\tau|^\delta).$$
By Lemma \ref{tenen}, we conclude that for $\max(l_1,l_2) \leq T^{\xi-\epsilon}$, 
\begin{align*}
&\max(l_1,l_2) \sum_{n \leq T^\xi/\max (l_1,l_2)} d_{1/2}(l_1n) d_{1/2}(l_2n) \l_f(l_1n) \l_f(l_2n) \\
&= \frac{T^\xi}{\log (\frac{T^\xi}{\max(l_1,l_2)})^{3/4}} G(1;1/4,l_1,l_2) \frac{1}{\Gamma(\frac{1}{4})}  +O \left(\frac{T^\xi |\l_f(l_1l_2)|d_{1/2}(l_1l_2)M_G(l_1l_2)}{(\log T^\xi/\max(l_1,l_2))^{7/4}}\right).   
\end{align*}
It then follows by summation by parts, that whenever $\max(l_1,l_2) \leq T^{\xi-\epsilon}$, we can estimate
\begin{align}\label{inner}
&\sum_{n\leq T^\xi/\max(l_1,l_2)} \frac{d_{1/2}(l_1n)d_{1/2}(l_2n) \l_f(l_1n) \l_f(l_2n)}{n}\\
\nonumber&=4 \frac{1}{\Gamma\left(\frac{1}{4}\right)} G(1;\frac{1}{4},l_1,l_2) \log\left(\frac{T^\xi}{\max(l_1,l_2)}\right)^{1/4}+ O\left(\frac{ |\l_f(l_1l_2)|d_{1/2}(l_1l_2)M_G(l_1l_2)}{ (\log \frac{T^\xi}{\max(l_1,l_2)})^{3/4}}\right).
\end{align}

\subsection{The $l_i$ and $g$ sums} We let $Z= \exp((\log N)^{2/3})$ and consider first the contribution from the main term above when $\max(l_1,l_2)< Z$. Namely, we estimate
\begin{align*}
\sum_{g \leq T^{1-3\xi}} r(g)^2 \l_f(g)^2 \sum_{\substack{l_1,l_2 \leq T^{1-3\xi}/g \\ \max(l_1,l_2)< Z \\ (l_1,l_2)=(l_1l_2,g)=1}} \frac{r(l_1l_2)}{(l_1l_2)^{1/2}} \l_f(l_1l_2)^2 d_{1/2}(l_1l_2) H(l_1l_2) \left(\log T\right)^{1/4}, 
\end{align*}
where $H(l)$ is a non-negative multiplicative function supported on squarefree integers, satisfying (\ref{mp}) on primes, possibly with a different constant. By Lemma \ref{r2g2} we thus estimate
\begin{align*}
&(\log T)^{1/4} \sum_{\substack{l_1,l_2 < Z \\ (l_1,l_2)=1}} \frac{r(l_1l_2)}{\sqrt{l_1l_2}} \l_f(l_1l_2)^2d_{1/2}(l_1l_2) H(l_1l_2) \sum_{\substack{g\leq \frac{T^{1-2\xi}}{\max(l_1,l_2)} \\ (g,l_1l_2)=1}} r(g)^2\l_f(g)^2\\
& \sim (\log T)^{1/4} \sum_{\substack{ l_1,l_2<Z\\ (l_1,l_2)=1}} \frac{r(l_1l_2)}{\sqrt{l_1l_2}} \l_f(l_1l_2)^2 d_{1/2}(l_1l_2) H(l_1l_2) \prod_{p\nmid l_1l_2} (1+r^2(p)\l_f(p)^2)\\
&= (\log T)^{1/4} \prod_p (1+r(p)^2\l_f(p)^2)\sum_{\substack{l_1,l_2 < Z\\ (l_1,l_2)=1}} \frac{r(l_1l_2)}{\sqrt{l_1l_2}} \l_f(l_1l_2)^2 d_{1/2}(l_1l_2) \tilde{H}(l_1l_2) ,
\end{align*}
where $\tilde{H}(l)$ is a non-negative, multiplicative function, absolutely bounded on primes and satisfying
$$\tilde{H}(p)= \frac{H(p)}{(1+r(p)^2\l_f(p)^2)}.$$
We make the change of variables $l=l_1l_2$ to reduce our estimation to that of
\begin{align*}
\sum_{l<Z} \frac{r(l)}{\sqrt{l}}\l_f(l)^2 \tilde{H}(l)&\sim \prod_p\left(1+\frac{r(p)}{\sqrt{p}} \l_f(p)^2\tilde{H}(p)\right)\sim \prod_p\left(1+\frac{r(p)}{\sqrt{p}} \l_f(p)^2\right) ,
\end{align*}
by Lemma \ref{rr2g}. The contribution from the tail $\max(l_1,l_2) \geq Z$ is bounded by 
\begin{align*}
&\log T \sum_{g\leq T^{1-3\xi}} r(g)^2 \l_f(g)^2 \sum_{l_1\leq T^{1-3\xi}/g} \frac{r(l_1) \l_f(l_1)^2}{\sqrt{l_1}} \sum_{Z<l_2 \leq T^{1-3\xi}/g} \frac{r(l_2)\l_f(l_2)^2}{\sqrt{l_2}}\\
&\ll \log T \exp(-(\log N)^{2/3-\epsilon})\prod_p(1+r(p)^2\l_f^2(p)) \prod_p\left(1+\frac{r(p)\l_f(p)^2}{\sqrt{p}}\right),
\end{align*}
by Lemma \ref{rr2g}, which is negligible. We are only left with estimating the contribution coming from the error term in (\ref{inner}), with $\max(l_1,l_2) < Z$. We thus care to bound 
\begin{align*}
&  (\log T)^{-3/4} \sum_{\substack{l_1,l_2<Z \\ (l_1,l_2)=1}} \frac{r(l_1l_2)}{\sqrt{l_1l_2}} \l_f(l_1l_2)^2d_{1/2}(l_1l_2)M_G(l_1l_2)\sum_{\substack{g\leq \frac{T^{1-3\xi}}{\max(l_1,l_2)} \\ (g,l_1l_2)=1}} r(g)^2\l_f(g)^2\\
&\sim (\log T)^{-3/4} \sum_{\substack{l_1,l_2<Z \\ (l_1,l_2)=1}} \frac{r(l_1l_2)}{\sqrt{l_1l_2}} \l_f(l_1l_2)^2d_{1/2}(l_1l_2) M_G(l_1l_2)\prod_{p \nmid l_1l_2} (1+r(p)^2\l_f(p)^2)\\
&\sim (\log T)^{-3/4} \prod_p (1 + r(p)^2\l_f(p)^2) \prod_p \left(1+\frac{r(p)M_G(p)\l_f^2(p)}{\sqrt{p} (1+r(p)^2\l_f(p)^2)}\right)
\end{align*}
which is negligible.
Putting all of the estimates together, we obtain (\ref{l7}).

\end{proof}
We conclude this section by computing the normalizing weight. 
\begin{prop}\label{prop2}
We may estimate the normalizing weight,
\begin{align*}
NW &\asymp (\log T)^{1/4} \prod_p(1+r(p)^2\l_f(p)^2) \prod_p \left(1+\frac{r(p)}{\sqrt{p}} \l_f(p)^2\right) I_T.
\end{align*}
\end{prop}
\begin{proof}
We denote by $\Gamma_\epsilon$ the contour defined by the line $\Re(s)=1/2+\epsilon$ clockwards and $\Re(s)= -1/2-\epsilon$ anticlockwards, so that up to negligible error, we have 
\begin{align*}
NW & \sim \frac{1}{2\pi i} \int_{\Gamma_\epsilon} R(s)R(-s) \frac{\Delta'(s)}{\Delta(s)-e^{2i\theta}} \frac{\mathrm{d}s}{\cos\left(\frac{iT-s}{H}\right)}.
\end{align*}
The integral on $\Re(s)= \frac{1}{2}+\epsilon$ is negligible by (\ref{third}). On $\Re(s) = -\frac{1}{2} - \epsilon$, we substitute $s\mapsto -s$ and thus need to estimate 
\begin{align*}
&\int_{\Re(s)= \frac{1}{2} + \epsilon} R(s)R(-s) \left(-\frac{\Delta'(s)}{\Delta(s)} \frac{1}{1-e^{2i\theta} \Delta(s)} \right) \frac{\mathrm{d}s}{\cos\left(\frac{iT+s}{H}\right)} \\
& = \sum_{k=0}^\infty \int_{\Re(s)=\frac{1}{2}+\epsilon} R(s)R(-s) \left(-\frac{\Delta'(s)}{\Delta(s)} \Delta^k(s) e^{2ik\theta}\right) \frac{\mathrm{d}s}{\cos\left(\frac{iT+s}{H}\right)} \\
&= - \int_{\Re(s)=\frac{1}{2}+\epsilon} R(s)R(-s) \frac{\Delta'(s)}{\Delta(s)} \frac{\mathrm{d}s}{\cos\left(\frac{iT+s}{H}\right)}\\
&\hspace{4 mm} - \int_{\Re(s)=\frac{1}{2}+\epsilon} R(s)R(-s) \Delta'(s)\frac{e^{2i\theta}}{1-e^{2i\theta}\Delta(s)} \frac{\mathrm{d}s}{\cos\left(\frac{iT+s}{H}\right)}\\
&= -\sum_{m,n\leq T^{1-2\xi}} a_ma_n \int_{\Re(s)=\frac{1}{2}+\epsilon} \left(\frac{m}{n}\right)^s \frac{\Delta'(s)}{\Delta(s)} \frac{\mathrm{d}s}{\cos\left(\frac{iT+s}{H}\right)}\\
&\hspace{4 mm}-e^{2i\theta} \sum_{m,n\leq T^{1-2\xi}} a_ma_n \int_{\Re(s)= \frac{1}{2}+\epsilon} \left(\frac{m}{n}\right)^s \frac{\Delta'(s)}{\Delta(s)} \frac{\Delta(s)}{1-e^{2i\theta}\Delta(s)} \frac{\mathrm{d}s}{\cos\left(\frac{iT+s}{H}\right)}.
\end{align*}
Using Lemma \ref{prelimlemma} and Lemma \ref{an2}, we conclude that
\begin{align*}
NW & \asymp (\log T)^{1/4} \prod_p(1+r(p)^2\l_f(p)^2) \prod_p \left(1+\frac{r(p)}{\sqrt{p}} \l_f(p)^2 \right)I_T.
\end{align*}

\end{proof}

\section{The unsigned moment}
We denote by $\bold{E}_{w_{T,\theta}}$ the expectation over $T_{\theta}$ with respect to the measure $w_{T,\theta}$ and wish to give a lower bound to 
\begin{align}\label{pm}
\nonumber&\hbox{NW} .\bold{E}_{w_{T,\theta}}\left[\left|L\left(f,\frac{1}{2}+it\right)\right|\right]  \geq \hbox{NW}. \left|\bold{E}_{w_{T,\theta}}\left[L\left(f,\frac{1}{2}-it\right)\frac{A_{1/2}\left(\frac{1}{2}+it\right)^2}{\left|A_{1/2}\left(\frac{1}{2}+it\right)\right|^2}\right]\right|\\
&\sim \frac{1}{2\pi}\left| \int_{\Gamma_\epsilon} L\left(f,\frac{1}{2}-s\right) A_{1/2}\left(\frac{1}{2}+s\right)^2 R^*(s)R^*(-s) \frac{\Delta'(s)}{\Delta(s)-e^{2i\theta}} \frac{\mathrm{d}s}{\cos\left(\frac{iT-s}{H}\right)}\right|\\
\nonumber&= \frac{1}{2\pi} \left| \int_{\Re(s) = 1/2+\epsilon} \cdots + \int_{\Re(s)= -1/2-\epsilon}\cdots \right|. 
\end{align}
\subsection{Contribution from the integral along the line $\Re(s) = \frac{1}{2}+\epsilon$ }We show that the contribution from this term is negligible. We first note that by Mellin inversion, for a smooth $\phi:\R \rightarrow [0,1]$ compactly supported in $[-1,1]$ such that $\phi \equiv 1$ in a neighborhood of $0$, we have uniformly in $\{s=\sigma+it: T/2\leq t \leq 2T, 0\leq \sigma \leq 2\}$, and for all $\epsilon >0, A>0$,
$$L(f,s)= \sum_{n\geq 1} \frac{\l_f(n)}{n^s} \phi\left(\frac{n}{T^{2+\epsilon}}\right) + O(T^{-A}).$$ 
By the definition of $\Delta(s)$, the integral becomes
\begin{small}
\begin{align*}
&\int_{\Re(s)=\frac{1}{2}+\epsilon} L\left(f,\frac{1}{2}+s\right) A_{1/2}\left(\frac{1}{2}+s\right)^2 R^*(s)R^*(-s) \frac{\Delta'}{\Delta}(s) \frac{1}{\Delta(s)-e^{2i\theta}} \frac{\mathrm{d}s}{\cos\left(\frac{iT-s}{H}\right)}\\
&= \sum_{n\geq 1} \frac{\l_f(n)}{n^{1/2}} \phi\left(\frac{n}{T^{2+\epsilon}}\right)\sum_{l_1,l_2 <T^{1-3\xi}} \sum_{m_1,m_2<T^\xi} \frac{d_{1/2}(m_1)d_{1/2}(m_2)\l_f(m_1)\l_f(m_2)}{(m_1m_2)^{1/2}} \\
&\times   r(l_1)r(l_2)\l_f(l_1)\l_f(l_2)\int_{\Re(s)=\frac{1}{2}+\epsilon} \left(\frac{l_2}{nm_1m_2l_1}\right)^s \frac{\Delta'}{\Delta}(s) \frac{1}{\Delta(s)-e^{2i\theta}} \frac{\mathrm{d}s}{\cos\left(\frac{iT-s}{H}\right)}  +O(T^{-A}).
\end{align*}
\end{small}
We write
\begin{equation*}
(\Delta(s)-e^{2i\theta})^{-1}= -e^{-2i\theta}- e^{-4i\theta} \frac{\Delta(s)}{1-e^{-2i\theta}\Delta(s)} ,
\end{equation*}
and the contribution of the second term to the $s$-integral is 
\begin{align*}
\int_{\Re(s)=\frac{1}{2}+\epsilon} \left(\frac{l_2}{nm_1m_2l_1}\right)^s \frac{\Delta'}{\Delta}(s) \frac{ \Delta(s)}{1-e^{-2i\theta}\Delta(s)} \frac{\mathrm{d}s}{\cos\left(\frac{iT-s}{H}\right)},
\end{align*}
which by (\ref{third}) is negligible. It remains to bound the contribution of the first term above, which is 
\begin{align*}
&\frac{1}{2\pi i}\int_{\Re(s)=\frac{1}{2}+\epsilon} \left(\frac{l_2}{nm_1m_2l_1}\right)^s \frac{\Delta'}{\Delta}(s) \frac{\mathrm{d}s}{\cos\left(\frac{iT-s}{H}\right)} = -\frac{1}{4\pi}\delta_{l_2=nm_1m_2l_1} I_T + O(T^{-A}),
\end{align*}
by (\ref{second}). We therefore just need to estimate 
\begin{small}
\begin{align*}
&\sum_{l_1,l_2<T^{1-3\xi}}\frac{\sqrt{l_1}r(l_1)r(l_2)\l_f(l_1)\l_f(l_2)}{\sqrt{l_2}}\sum_{\substack{m_1,m_2<T^\xi \\ nm_1m_2= \frac{l_2}{l_1}}} \l_f(n) d_{1/2}(m_1)d_{1/2}(m_2)\l_f(m_1)\l_f(m_2) I_T\\
&=  \sum_{l_1,l_2<T^{1-3\xi}}\frac{\l_f(l_2)^2\sqrt{l_1}r(l_1)r(l_2)}{\sqrt{l_2}}\sum_{\substack{m_1,m_2<T^\xi \\ nm_1m_2= \frac{l_2}{l_1}}}  d_{1/2}(m_1)d_{1/2}(m_2) I_T,
\end{align*}
\end{small}
After making a change of variables $l_2=l_2/l_1$, we thus estimate
\begin{align*}
&\sum_{l_1 < T^{1-3\xi}} r(l_1)^2 \l_f(l_1)^2 \sum_{\substack{l_2\leq \frac{T^{1-3\xi}}{l_1} \\ (l_1,l_2)=1}} \frac{\l_f(l_2)^2r(l_2)}{\sqrt{l_2}}\sum_{\substack{m_1,m_2<T^\xi \\ nm_1m_2=l_2}} d_{1/2}(m_1)d_{1/2}(m_2) I_T\\
&\leq \sum_{l_1<T^{1-3\xi}} r(l_1)^2 \l_f(l_1)^2 \sum_{\substack{l_2\leq \frac{T^{1-3\xi}}{l_1}\\ (l_1,l_2)=1}} \frac{\l_f(l_2)^2r(l_2)}{\sqrt{l_2}} \sum_{m|l_2 < T^{2\xi}} d(m) d_{1/2}(m) I_T\\
&\leq  \sum_{l_1<T^{1-3\xi}} r(l_1)^2 \l_f(l_1)^2 \sum_{\substack{l_2\leq \frac{T^{1-3\xi}}{l_1}\\ (l_1,l_2)=1}} \frac{\l_f(l_2)^2r(l_2)}{\sqrt{l_2}} d(l_2)I_T\\
&\ll \prod_p(1+r(p)^2\l_f(p)^2) \left(\prod_p\left(1+\frac{r(p)}{\sqrt{p}} \l_f(p)^2\right)\right)^2 I_T.
\end{align*}
Dividing by the normalizing weight, using Proposition \ref{prop2}, we see that the contribution from $\Re(s)=1/2+\epsilon$ in (\ref{pm}) is bounded by
$$\ll (\log T)^{-1/4} \prod_p\left(1+\frac{r(p)}{\sqrt{p}} \l_f(p)^2\right),$$
which is smaller than (\ref{unsignedmoment}) by a factor of $\log T$.
\subsection{The main term} The integral along the line $\Re(s)=-1/2-\epsilon$  contributes to (\ref{pm}) as a main term. We make the change of variables $s\rightarrow -s$, and estimate 
\begin{small}
\begin{align*}
&\int_{\Re(s)=\frac{1}{2}+\epsilon} L\left(f,\frac{1}{2}+s\right) A_{1/2}\left(\frac{1}{2}-s\right)^2 R^*(s)R^*(-s) \frac{\Delta'(-s)}{\Delta(-s)-e^{2i\theta}} \frac{\mathrm{d}s}{\cos\left(\frac{iT+s}{H}\right)}\\
&= \int_{\Re(s)=\frac{1}{2}+\epsilon} L\left(f,\frac{1}{2}+s\right) A_{1/2}\left(\frac{1}{2}-s\right)^2 R^*(s)R^*(-s) \frac{\Delta'}{\Delta}(s) \frac{1}{1-e^{2i\theta} \Delta(s)} \frac{\mathrm{d}s}{\cos\left(\frac{iT+s}{H}\right)}\\
&= \sum_{n\geq 1} \frac{\l_f(n)}{n^{1/2}} \phi\left(\frac{n}{T^{2+\epsilon}}\right)\sum_{l_1,l_2 < T^{1-3\xi}}  \sum_{m_1,m_2 < T^\xi} \frac{d_{1/2}(m_1)d_{1/2}(m_2) \l_f(m_1)\l_f(m_2)}{(m_1m_2)^{1/2}}  \\
& \times r(l_1)r(l_2) \l_f(l_1)\l_f(l_2) \int_{\Re(s)=\frac{1}{2}+\epsilon} \left(\frac{m_1m_2l_2}{nl_1}\right)^s \frac{\Delta'}{\Delta}(s) \frac{1}{1-e^{2i\theta}\Delta(s)} \frac{\mathrm{d}s}{\cos\left(\frac{iT+s}{H}\right)}.
\end{align*}
\end{small}
We write
$$(1-e^{2i\theta}\Delta(s))^{-1}= 1+\frac{e^{2i\theta}\Delta(s)}{1-e^{2i\theta}\Delta(s)},$$ 
 and by the same observation as before, only the contribution from the first term above is non-negligible. By (\ref{second}), this term yields, up to negligible error term,
\begin{gather}
\nonumber\sum_{l_1,l_2 < T^{1-3\xi}} r(l_1)r(l_2) \l_f(l_1)\l_f(l_2) \sum_{\substack{m_1,m_2 < T^\xi \\ m_1m_2l_2=nl_1}} \frac{\l_f(n)d_{1/2}(m_1)d_{1/2}(m_2)\l_f(m_1)\l_f(m_2)}{(nm_1m_2)^{1/2}} I_T\\
\label{sg} =  \sum_{g\leq T^{1-3\xi}} S(g) I_T,
\end{gather}
where $S(g)$ is defined as
\begin{small}
\begin{align*}
&\sum_{\substack{l_1,l_2 \leq \frac{T^{1-3\xi}}{g}\\ (l_1,l_2)=(l_1l_2,g)=1}} r(l_1)r(l_2)\l_f(l_1)\l_f(l_2) \sum_{\substack{m_1,m_2<T^\xi\\ m_1m_2l_2=nl_1}} \frac{\l_f(n)d_{1/2}(m_1)d_{1/2}(m_2) \l_f(m_1)\l_f(m_2)}{\sqrt{nm_1m_2}}.
\end{align*}
\end{small}
We let $l_{11}= (l_1,m_1), l_{12}=l_1/l_{11}$ and $m_1:= m_1/l_{11}, m_2:= m_2/l_{12}$, so that
\begin{align*}
S(g) &= \sum_{\substack{l_1,l_2 \leq \frac{T^{1-3\xi}}{g}\\ (l_1,l_2)=(l_1l_2,g)=1}} \frac{r(l_1)r(l_2)\l_f(l_1)\l_f(l_2)}{\sqrt{l_1l_2}} \sum_{l_{11}l_{12}=l_1}\\
&\times \sum_{l_{11}m_1,l_{12}m_2<T^\xi} \frac{\l_f(m_1m_2l_2)d_{1/2}(l_{11}m_1)d_{1/2}(l_{12}m_2)\l_f(l_{11}m_1)\l_f(l_{12}m_2)}{m_1m_2}.
\end{align*}
We will estimate the outer sum by repeated use of Lemma \ref{tenen}. We first evaluate the $m_1$-sum and then the $m_2$-sum.
\subsubsection{The $m_1$-sum}
Writing $l$ for $l_{11}$ and $m$ for $m_1$, we study the series

\begin{align*}
\sum_{m=1}^\infty \frac{\l_f(m_2l_2m)d_{1/2}(lm)\l_f(lm)}{m^s} &= G_1(s;m_2,l_2,l) \prod_p\left(\sum_{k=0}^\infty \l_f(p^k)^2d_{1/2}(p^k)p^{-ks}\right),
\end{align*}
where
$$G_1(s;m_2,l_2,l):= \prod_{p|m_2l_2l} \frac{\sum_{k=0}^\infty \l_f(p^{\nu_p(m_2l_2)+k})d_{1/2}(p^{\nu_p(l)+k})\l_f(p^{\nu_p(l)+k})p^{-ks}}{\sum_{k=0}^\infty \l_f(p^k)^2d_{1/2}(p^k)p^{-ks}}.$$
We wish to relate our Euler product to $L^{1/2}(\hbox{sym}^2f,s)$. We have
\begin{align*}
L^{1/2}(f\times f,s) =& \prod_p(1-p^{-s})^{-1}(1-\alpha_p^2p^{-s})^{-1/2}(1-\beta_p^{2}p^{-s})^{-1/2}\\
=&\prod_p \left(\sum_{k=0}^\infty p^{-ks}\right) \left(\sum_{k=0}^\infty \left(\begin{array}{c} k-\frac{1}{2}\\ k \end{array}\right) \alpha_p^{2k}p^{-ks}\right)\\
& \times \left(\sum_{k=0}^\infty \left(\begin{array}{c} k-\frac{1}{2} \\ k \end{array}\right) \beta_p^{2k}p^{-ks}\right)\\
=& \prod_p\left(\sum_{k=0}^\infty b(p^k)p^{-ks}\right),
\end{align*}
where $b$ is a non-negative multiplicative function such that
$$b(p)= \frac{\l_f(p)^2}{2} = d_{1/2}(p) \l_f^2(p).$$
Writing $s=\sigma +i\tau$, we thus have
\begin{equation}\label{prodsquare}
\prod_p\left(\sum_{k=0}^\infty \l_f(p^k)^2d_{1/2}(p^k)p^{-ks}\right) = \zeta(s)^{1/2}L^{1/2}(\hbox{sym}^2f,s) B(s),
\end{equation}
where $B(s)$ is a bounded holomorphic function in the region $\sigma> 1- c/\log(2+|\tau|)$. We write
$$\sum_{m=1}^\infty \frac{\l_f(m_2l_2m)d_{1/2}(lm)\l_f(lm)}{m^s}= G_1\left(s;\frac{1}{2},m_2,l_2,l\right) \zeta(s)^{1/2},$$
where 
$$G_1\left(s;\frac{1}{2},m_2,l_2,l\right) := L^{1/2}(\hbox{sym}^2f,s)B(s) G_1(s;m_2,l_2,l).$$
We define $M_1(m_2,l_2,l)$ to be
$$   \prod_{p|m_2l_2l} \sup_{\sigma > 1 -\frac{c}{\log(2+|\tau|)}} \left|\frac{\sum_{k=0}^\infty |\l_f(p^{\nu_p(m_2l_2)+k})d_{1/2}(p^{\nu_p(l)+k})\l_f(p^{\nu_p(l)+k})|p^{-ks}}{\sum_{k=0}^\infty \l_f(p^k)^2d_{1/2}(p^k)p^{-ks}}\right|,$$
and deduce by Lemma \ref{tenen} the following lemma.
\begin{lemma}\label{lemmasum1}
For $l \leq T^{\xi-\epsilon}$, we have
\begin{align*}
\sum_{m<\frac{T^\xi}{l}} \frac{\l_f(m_2l_2m)d_{1/2}(lm)\l_f(lm)}{m} & = \frac{2+o(1)}{\Gamma\left(\frac{1}{2}\right)} \left(\log \frac{T^\xi}{l}\right)^{1/2} G_1\left(1;\frac{1}{2},m_2,l_2,l\right)\\
&+ O\left(\frac{M_1}{\log^{1/2} T}\right).
\end{align*}
\end{lemma}
\subsubsection{The $m_2$-sum}
We now evaluate the contribution of the main term of Lemma \ref{lemmasum1} and study the associated Dirichlet series
\begin{align*}
&\sum_{m_2} \frac{d_{1/2}(l_{12} m_2) \l_f(l_{12}m_2) G_1\left(1;m_2,l_2,l_{11}\right)}{m_2^s} \\ 
&= G_2(s; l_{11},l_{12},l_2) \prod_{p} \frac{\displaystyle \sum_{k,k'=0}^\infty \frac{d_{1/2}(p^k) \l_f(p^k) \l_f(p^{k+k'}) d_{1/2}(p^{k'}) \l_f(p^{k'})}{ p^{ks+k'}}}{\sum_{k=0}^\infty \l_f(p^k)^2d_{1/2}(p^k) p^{-k}} ,
\end{align*}
where $G_2(s; l_{11},l_{12},l_2)$ is defined as
$$\prod_{p| l_2l_{11}l_{12}} G_{2,p}(s; l_{11},l_{12},l_2),$$
and $G_{2,p}$ is given by
\begin{small}
$$ \frac{\displaystyle\sum_{k,k'=0}^\infty d_{\frac{1}{2}} (p^{\nu_p(l_{12})+k}) \l_f(p^{\nu_p(l_{12})+k}) \l_f(p^{\nu_p(l_2)+k+k'}) d_{\frac{1}{2}}(p^{\nu_p(l_{11})+k'})\l_f(p^{\nu_p(l_{11})+k'})p^{-ks-k'}}{\displaystyle\sum_{k,k'=0}^\infty d_{\frac{1}{2}}(p^k) \l_f(p^k) \l_f(p^{k+k'}) d_{\frac{1}{2}}(p^{k'}) \l_f(p^{k'}) p^{-ks-k'}}.$$
We note that the prime factors of $l_1, l_2$ are, by the support of $r$, large enough so that the denominator above does not vanish.
\end{small}
\begin{claim}
Let $s=\sigma +i\tau$; there exists a function, $C(s)$, bounded and holomorphic in the region $\sigma>1-c/\log (2+|\tau|)$ such that
 $$\prod_{p} \frac{\displaystyle \sum_{k,k'=0}^\infty \frac{d_{\frac{1}{2}}(p^k) \l_f(p^k) \l_f(p^{k+k'}) d_{\frac{1}{2}}(p^{k'}) \l_f(p^{k'})}{ p^{ks+k'}}}{\displaystyle\sum_{k=0}^\infty \l_f(p^k)^2 d_{\frac{1}{2}}(p^k) p^{-k}}= \zeta^{1/2}(s) L^{1/2}(\hbox{sym}^2f,s) C(s).$$
\end{claim}
\begin{proof}
We have 
\begin{align*}
&\prod_{p} \frac{\sum_{k,k'=0}^\infty d_{1/2}(p^k) \l_f(p^k) \l_f(p^{k+k'}) d_{1/2}(p^{k'}) \l_f(p^{k'}) p^{-ks-k'}}{\sum_{k=0}^\infty \l_f(p^k)^2d_{1/2}(p^k) p^{-k}}\\
&= \prod_p \left(1 + d_{1/2}(p) \l_f(p)^2 p^{-s} + O(p^{-s-\frac{1}{2}}) \right),
\end{align*}
and the claim follows immediately.
\end{proof} 
By the above claim, we have
\begin{align*}
\sum_{m_2} \frac{d_{1/2}(l_{12}m_2)\l_f(l_{12}m_2)G_1(1;m_2,l_2,l_{11})}{m_2^s} = \zeta^{1/2}(s) G_2\left(s;\frac{1}{2},l_{11},l_{12},l_2\right),
\end{align*}
where
$$G_2\left(s;\frac{1}{2},l_{11},l_{12},l_2\right) = G_2(s;l_{11},l_{12},l_2) L^{1/2}(\hbox{sym}^2f,s) C(s).$$
We let 
$$M_{2}(l_{11},l_{12},l_2)= \prod_{p|l_{11}l_{12}l_2} \sup_{\sigma> 1-\frac{c}{\log(2+|\tau|)}} |M_{2,p}(s)|$$
where  $M_{2,p}(s)$ is given by
\begin{small}
 \begin{align*}
  \frac{\displaystyle \sum_{k,k'=0}^\infty \left|d_{\frac{1}{2}}(p^{\nu_p(l_{12})+k})\l_f(p^{\nu_p(l_{12})+k})\l_f(p^{\nu_p(l_2)+k+k'})d_{\frac{1}{2}}(p^{\nu_p(l_{11})+k'}) \l_f(p^{\nu_p(l_{11})+k'})\right|p^{-ks-k'}}{\displaystyle \sum_{k,k'=0}^\infty |d_{\frac{1}{2}}(p^k)\l_f(p^k)\l_f(p^{k'})\l_f(p^{k+k'}) d_{\frac{1}{2}}(p^{k'})|p^{-ks-k'}}.
\end{align*}
\end{small}
We note that by the parity of $\nu_p(l_{12})+k, \nu_p(l_2)+k+k',$ and $\nu_p(l_{11})+k'$, we have
$$M_2(l_{11},l_{12},l_2) \leq d_{1/2}(l_1) |\l_f(l_1l_2)| M_2(l_1l_2),$$
where $M_2(l)$ is a positive multiplicative function supported on squarefree integers and satisfying 
\begin{equation}\label{m2b}
M_2(p)^{\pm 1} \leq  (1 + C_2 p^{-\delta_2}),
\end{equation}
for some absolute constant $C_2$ and some $\delta_2 > 0$. By Lemma \ref{tenen}, we obtain for $l_{12}< T^{\xi-\epsilon},$
\begin{small}
\begin{align*}
\sum_{m_2<\frac{T^\xi}{l_{12}}} \frac{d_{1/2}(l_{12}m_2)\l_f(l_{12}m_2)G(1;m_2,l_2,l_{11})}{m_2} &=  \frac{2+o(1)}{\Gamma\left(\frac{1}{2}\right)} G_2\left(1;\frac{1}{2},l_{11},l_{12},l_2\right) \log^{1/2}\left(\frac{T^\xi}{l_{12}}\right)  \\
&  +O\left(\frac{d_{1/2}(l_1) |\l_f(l_1l_2)|M_2(l_1l_2)}{\log^{1/2}T}\right).
\end{align*}
\end{small}
We may control the contribution from the error term in Lemma \ref{lemmasum1} similarly. Namely, with $s=\sigma+i\tau$ and $z=\delta+i \gamma$, we let 
$$M_{3}(l_{11},l_{12},l_2) :=  \prod_{p|l_{11}l_{12}l_2} \sup_{\sigma > 1- \frac{c}{\log (2+|\tau|)}} \frac{\displaystyle\sum_{k=0}^\infty \sup_{\delta> 1-\frac{c}{\log(2+|\gamma|)}} |G_{3,p}(k;s,z,l_{11},l_{12},l_2)|}{\displaystyle\sum_{k=0}^\infty \sup_{\delta > 1- \frac{c}{\log (2+|\gamma|)}} |G^\dagger_{3,p}(k;s,z,l_{11},l_{12},l_2)|},$$
where $G_{3,p}(k;s,z,l_{11},l_{12},l_2)$ is given by
\begin{small}
$$\frac{\displaystyle \sum_{k'=0}^\infty |\l_f(p^{\nu_p(l_2)+k+k'}) d_{1/2}(p^{\nu_p(l_{11})+k'}) \l_f(p^{\nu_p(l_{11})+k'}) d_{1/2}(p^{\nu_p(l_{12})+k})\l_f(p^{\nu_p(l_{12})+k})|p^{-ks-k'z}}{\displaystyle \sum_{k'=0}^\infty  \l_f(p^{k'})^2d_{1/2}(p^{k'})p^{-k'z}},$$
\end{small}
and $G_{3,p}^\dagger(k;s,z,l_{11},l_{12},l_2)$ is given by
$$\frac{\displaystyle \sum_{k'=0}^\infty |\l_f(p^{k+k'}) d_{1/2}(p^{k'}) \l_f(p^{k'}) d_{1/2}(p^k) \l_f(p^k)|p^{-ks-k'z}}{\displaystyle \sum_{k'=0}^\infty \l_f(p^{k'})^2d_{1/2}(p^{k'})p^{-k'z}}.$$
We note that we also have 
$$M_3(l_{11},l_{12},l_2) \leq d_{1/2}(l_1) |\l_f(l_1l_2)| M_3(l_1l_2),$$
where $M_3\geq M_2$ is a function satisfying (\ref{m2b}) possibly with a different constant. Using these to bound the contribution from the error term, we conclude the following lemma.
\begin{lemma}\label{dt}
For $l_{11}, l_{12}< T^{\xi-\epsilon},$ we have
\begin{align*}
&\sum_{l_{11}m_1,l_{12}m_2<T^\xi} \frac{\l_f(m_1m_2l_2)d_{1/2}(l_{11}m_1)d_{1/2}(l_{12}m_2)\l_f(l_{11}m_1)\l_f(l_{12}m_2)}{m_1m_2}  \\
&= \left(\frac{4+o(1)}{\Gamma\left(\frac{1}{2}\right)}\right)^2 \left(\log \frac{T^\xi}{l_{11}}\right)^{1/2} \left(\log \frac{T^\xi}{l_{12}}\right)^{1/2} G_2\left(1;\frac{1}{2},l_{11},l_{12},l_2\right)\\
& \hspace{6 cm}+O(d_{1/2}(l_1) |\l_f(l_1l_2)| M_3(l_1l_2)).
\end{align*}
\end{lemma}
\subsubsection{The $l_1$ and $l_2$ sums}
We let $Z= \exp((\log N)^{2/3})$ and note that by Lemma \ref{rr2g} the contribution from $l_1,l_2 \geq Z$ to $S(g)$ is negligible. We first consider the contribution from the main term in Lemma \ref{dt} to (\ref{sg}), yielding
\begin{align*}
&\log T \sum_{\substack{l_1,l_2 < Z\\ (l_1,l_2)=1}} \frac{r(l_1l_2) \l_f(l_1l_2)}{\sqrt{l_1l_2}} \sum_{l_{11}l_{12}=l_1} G_2\left(1;\frac{1}{2},l_{11},l_{12},l_2\right) \sum_{\substack{g\leq \frac{T^{1-3\xi}}{\max(l_1,l_2)}\\ (g,l_1l_2)=1}} r(g)^2\l_f(g)^2 I_T \\
& \sim \log T \sum_{\substack{l_1,l_2 < Z\\ (l_1,l_2)=1}} \frac{r(l_1l_2) \l_f(l_1l_2)}{\sqrt{l_1l_2}} \sum_{l_{11}l_{12}=l_1} G_2\left(1;\frac{1}{2},l_{11},l_{12},l_2\right) \prod_{p\nmid l_1l_2} (1+r(p)^2\l_f(p)^2)I_T\\
&\sim  \log T \sum_{\substack{l_1,l_2 < Z\\ (l_1,l_2)=1}} \frac{r(l_1l_2) \l_f(l_1l_2)}{\sqrt{l_1l_2}} d(l_1) G_2(l_1,l_2)   \prod_p (1+r(p)^2\l_f(p)^2)I_T,
\end{align*}
where up to a constant $G_2(l_1,l_2)$ is given by
\begin{small}
\begin{align*}
& \prod_{p|l_1l_2} \frac{\displaystyle \sum_{k,k'=0}^\infty d_{\frac{1}{2}}(p^{\nu_p(l_1)+k})\l_f(p^{\nu_p(l_1)+k}) \l_f(p^{\nu_p(l_2)+k+k'})d_{\frac{1}{2}}(p^{k'})\l_f(p^{k'})p^{-k-k'}}{(1+r(p)^2\l_f(p)^2)\displaystyle\sum_{k,k'=0}^\infty d_{\frac{1}{2}}(p^k)\l_f(p^k)\l_f(p^{k+k'})d_{\frac{1}{2}}(p^{k'})\l_f(p^{k'})p^{-k-k'}}.
\end{align*}
\end{small}
Since for any $l_1, l_2, k, k'$, one of $\nu_p(l_1)+k, \nu_p(l_2)+k+k'$ and $k'$ must be odd, we may factorize $\l_f(l_1l_2)$ and obtain 
$$\l_f(l_1l_2)G_2(l_1,l_2) \geq d_{1/2}(l_1)\l_f(l_1l_2)^2 G(l_1l_2),$$
where $G$ is some multiplicative function supported on squarefree integers and satisfying
$$G(p)^{\pm 1} \leq (1+C_3p^{-\delta_3}) ,$$
for some absolute constant $C_3$ and some $\delta_3>0$. From Lemma \ref{rr2g} we have the following sequence of estimates:
\begin{equation}\label{estmt}
\sum_{l<Z} \frac{d(l)r(l)\l_f(l)^2}{\sqrt{l}}G(l) \sim \prod_p \left(1+\frac{2r(p)\l_f(p)^2}{\sqrt{p}} G(p)\right) \sim\prod_p \left(1+\frac{2r(p)\l_f(p)^2}{\sqrt{p}} \right) .
\end{equation}
The lower bound (\ref{unsignedmoment}) follows after dividing by the Normalizing Weight. 

We now consider the contribution from the error terms in Lemma \ref{dt}. We estimate
\begin{gather}\label{estet}
\sum_{\substack{l_1,l_2 <Z \\ (l_1,l_2)=1}} \frac{r(l_1l_2)|\l_f(l_1l_2)|}{\sqrt{l_1l_2}} \sum_{l_{11}l_{12}=l_1} d_{1/2}(l_1) |\l_f(l_1l_2)| M_{3}(l_1l_2) \prod_{p\nmid l_1l_2} (1+r(p)^2\l_f(p)^2)\\
\nonumber = \sum_{\substack{l_1,l_2 <Z \\ (l_1,l_2)=1}} \frac{r(l_1l_2)|\l_f(l_1l_2)|^2}{\sqrt{l_1l_2}}  \tilde{M}_3(l_1,l_2) \prod_p (1+r(p)^2\l_f(p)^2),
\end{gather}
where $\tilde{M}_3$ is a multiplicative function supported on squarefree integers defined on primes by
$$\tilde{M}_3(p) = \frac{M_3(p)}{(1+r(p)^2\l_f(p)^2)}.$$
We then may evaluate (\ref{estet}) as in (\ref{estmt}), however the contribution from this term is smaller as we save a factor of $\log T$ in the error term of Lemma \ref{dt}. 
\section{The signed moment}
In this section we prove (\ref{signedmoment}), by studying
\begin{small}
\begin{align*}
&\hbox{NW}.\bold{E}_{w_{T,\theta}}\left[L\left(f,\frac{1}{2}+it\right)\right]\\
&\sim \int_{\Gamma_\epsilon} L\left(f,\frac{1}{2}+s\right) A_{\frac{1}{2}}\left(\frac{1}{2}+s\right)A_{\frac{1}{2}}\left(\frac{1}{2}-s\right) R^*(s)R^*(-s) \frac{\Delta'(s)}{\Delta(s)-e^{2i\theta}} \frac{\mathrm{d}s}{\cos\left(\frac{iT-s}{H}\right)}.
\end{align*}
\end{small}
The contribution of the integral along the line $\Re(s)=1/2+\epsilon$ is
\begin{align*}
&-e^{-2i\theta}\sum_{n\geq 1} \frac{\l_f(n)}{n^{1/2}} \phi\left(\frac{n}{T^{2+\epsilon}}\right)\sum_{l_1,l_2<T^{1-3\xi}} \sum_{m_1,m_2<T^\xi} \frac{d_{1/2}(m_1)d_{1/2}(m_2)\l_f(m_1)\l_f(m_2)}{(m_1m_2)^{1/2}} \\
& \times r(l_1)r(l_2)\l_f(l_1)\l_f(l_2)  \int_{\Re(s)=\frac{1}{2}+\epsilon} \left(\frac{m_2l_2}{nm_1l_1}\right)^s \frac{\Delta'}{\Delta}(s) \frac{\Delta(s)}{1-e^{-2i\theta}\Delta(s)} \frac{\mathrm{d}s}{\cos\left(\frac{iT-s}{H}\right)},
\end{align*}
which by (\ref{third}) is negligible. We thus only care to estimate the integral along the line $\Re(s)=-1/2-\epsilon$. We make a change of variables $s\rightarrow -s$ and use the definition of $\Delta(s)$ to find
\begin{small}
\begin{align*}
& \int_{\Re(s)=\frac{1}{2}+\epsilon} L\left(f,\frac{1}{2}+s\right) A_{1/2}\left(\frac{1}{2}+s\right)A_{1/2}\left(\frac{1}{2}-s\right) R^*(s)R^*(-s)  \frac{\Delta'(-s)}{1-e^{2i\theta}\Delta(s)} \frac{\mathrm{d}s}{\cos\left(\frac{iT+s}{H}\right)}\\
&= \sum_{n\geq 1} \frac{\l_f(n)}{n^{1/2}} \phi\left(\frac{n}{T^{2+\epsilon}}\right) \sum_{m_1,m_2<T^\xi} \frac{d_{1/2}(m_1)d_{1/2}(m_2)\l_f(m_1)\l_f(m_2)}{(m_1m_2)^{1/2}} \\
&\times  \sum_{l_1,l_2<T^{1-2\xi}} r(l_1)r(l_2)\l_f(l_1)\l_f(l_2) \int_{\Re(s)=\frac{1}{2}+\epsilon} \left(\frac{m_2l_2}{nm_1l_1}\right)^s \frac{\Delta'(-s)}{1-e^{2i\theta}\Delta(s)} \frac{\mathrm{d}s}{\cos\left(\frac{iT+s}{H}\right)}.
\end{align*}
\end{small}
We write
$$(1-e^{2i\theta}\Delta(s))^{-1}= 1 + e^{2i\theta}\Delta(s) + \frac{e^{4i\theta}\Delta^2(s)}{1-e^{2i\theta}\Delta(s)}$$
to obtain the following three terms
\begin{align*}
\hbox{I}&:= \sum_{l_1,l_2<T^{1-3\xi}} r(l_1)r(l_2) \l_f(l_1)\l_f(l_2) \sum_{m_1,m_2<T^\xi} \frac{d_{1/2}(m_1)d_{1/2}(m_2)\l_f(m_1)\l_f(m_2)}{(m_1m_2)^{1/2}} \\
& \hspace{2 cm} \times \int_{\Re(s)=\frac{1}{2}+\epsilon} L\left(f,\frac{1}{2}+s\right) \left(\frac{m_2l_2}{m_1l_1}\right)^s \Delta'(-s) \frac{\mathrm{d}s}{\cos\left(\frac{iT+s}{H}\right)},
\end{align*}
\begin{align*}
\hbox{II}&:=e^{2i\theta} \sum_{n\geq1}\frac{\l_f(n)}{n^{1/2}}\phi\left(\frac{n}{T^{2+\epsilon}}\right)   \sum_{m_1,m_2 < T^\xi} \frac{d_{1/2}(m_1)d_{1/2}(m_2)\l_f(m_1)\l_f(m_2)}{(m_1m_2)^{1/2}}\\
& \times \sum_{l_1,l_2 < T^{1-2\xi}} r(l_1)r(l_2) \l_f(l_1)\l_f(l_2) \int_{\Re(s) = \frac{1}{2}+\epsilon} \left(\frac{m_2l_2}{nm_1l_1}\right)^s \frac{\Delta'}{\Delta}(s) \frac{\mathrm{d}s}{\cos\left(\frac{iT+s}{H}\right)},
\end{align*}
and 
\begin{align*}
\hbox{III}&:= e^{4i\theta} \sum_{n\geq 1} \frac{\l_f(n)}{n^{1/2}} \phi\left(\frac{n}{T^{2+\epsilon}}\right) \sum_{m_1,m_2<T^\xi} \frac{d_{1/2}(m_1)d_{1/2}(m_2) \l_f(m_1)\l_f(m_2)}{(m_1m_2)^{1/2}} \\
& \times \sum_{l_1,l_2 < T^{1-2\xi}} r(l_1)r(l_2) \l_f(l_1)\l_f(l_2) \int_{\Re(s)=\frac{1}{2}+\epsilon} \left(\frac{m_2l_2}{nm_1l_1}\right)^s  \frac{\Delta'(s)}{1-e^{2i\theta}\Delta(s)} \frac{\mathrm{d}s}{\cos\left(\frac{iT+s}{H}\right)}.
\end{align*}
We can see from (\ref{third}) that III is negligible, and we shall therefore focus solely on I and II. 
\subsection{Bounding II}
Using (\ref{second}), II is bounded up to negligible error term by
\begin{align*}
 & \sum_{g < T^{1-3\xi}} r(g)^2\l_f(g)^2 \sum_{\substack{l_1,l_2< T^{1-3\xi}/g \\ (l_1,l_2)=(l_1l_2,g)=1}} \l_f(l_1l_2)r(l_1l_2)\\
& \times  \sum_{\substack{m_1,m_2<T^\xi \\ nm_1l_1=m_2l_2}} \frac{\l_f(n)d_{1/2}(m_1)d_{1/2}(m_2)\l_f(m_1)\l_f(m_2)}{(nm_1m_2)^{1/2}} I_T.
\end{align*}
We let $l_{21}= (l_2,m_1), l_{22}= (l_2,n)$ and replace $m_1:=\frac{m_1}{l_{21}}, n:=\frac{n}{l_{22}}$ to reduce the problem to estimating
\begin{align*}
&\sum_{g<T^{1-3\xi}} r(g)^2\l_f(g)^2 \sum_{\substack{l_1,l_2 < T^{1-3\xi}/g \\ (l_1,l_2)=(l_1l_2,g)=1}} \frac{\l_f(l_1l_2) r(l_1l_2)}{\sqrt{l_1l_2}} \times \\
& \sum_{l_{21}l_{22}=l_2} \sum_{nm_1l_1,l_{21}m_1<T^\xi} \frac{\l_f(l_{22}n)d_{1/2}(l_{21}m_1)d_{1/2}(m_1nl_1) \l_f(l_{21}m_1)\l_f(m_1nl_1)}{nm_1}.
\end{align*}
We note that the innermost sum is bounded by
\begin{equation}\label{bsum}
\sum_{nl_1,l_{21}m_1<T^\xi} \frac{|\l_f(l_{22}n) d_{1/2}(l_{21}m_1)d_{1/2}(l_1m_1n)\l_f(l_{21}m_1)\l_f(m_1nl_1)|}{nm_1}.
\end{equation}
\subsubsection{Bounding (\ref{bsum})}
We study
\begin{small}
\begin{align*}
\sum_{n} \frac{|\l_f(l_{22}n)d_{1/2}(l_1m_1n)\l_f(m_1nl_1)|}{n^s}  &= \prod_p \left(\sum_{k=0}^\infty \l_f(p^k)^2d_{1/2}(p^k)p^{-ks}\right) G_3(s;l_1,m_1,l_{22}),
\end{align*}
\end{small}
where
\begin{small}
$$G_3(s;l_1,m_1,l_{22}):= \prod_{p|l_1m_1l_{22}} \frac{\displaystyle \sum_{k=0}^\infty |\l_f(p^{\nu_p(l_{22})+k})d_{\frac{1}{2}}(p^{\nu_p(l_1m_1)+k})\l_f(p^{\nu_p(l_1m_1)+k})|p^{-ks}}{\displaystyle\sum_{k=0}^\infty \l_f(p^k)^2d_{\frac{1}{2}}(p^k)p^{-ks}}.$$
\end{small}
By (\ref{prodsquare}), we conclude that 
$$\sum_{n} \frac{|\l_f(l_{22}n)d_{1/2}(l_1m_1n)\l_f(m_1nl_1)|}{n^s} = \zeta^{1/2}(s) G_3\left(s;\frac{1}{2}, l_1,m_1,l_{22}\right),$$
where
$$G_3\left(s;\frac{1}{2}, l_1,m_1,l_{22}\right)= G_3(s;l_1,m_1,l_{22}) L^{1/2}(\hbox{sym}^2f,s) B(s),$$
where $B(s)$ is given in (\ref{prodsquare}). Letting $M_3(l_1,m_1,l_{22})$ denote
$$\prod_{p | l_1m_1l_{22}} \sup_{\sigma> 1-c/\log (2+|\tau|)} \left|\frac{\sum_{k=0}^\infty |\l_f(p^{\nu_p(l_{22})+k})d_{\frac{1}{2}}(p^{\nu_p(l_1m_1)+k})\l_f(p^{\nu_p(l_1m_1)+k})|p^{-ks}}{\sum_{k=0}^\infty \l_f(p^k)^2d_{\frac{1}{2}}(p^k)p^{-ks}}\right|,$$
we use Lemma \ref{tenen} to conclude that for $l_1 < T^{\xi-\epsilon},$ we have 
\begin{align}\label{fsum}
\nonumber\sum_{l_1n<T^\xi} \frac{|\l_f(l_{22}n)d_{1/2}(l_1m_1n)\l_f(m_1nl_1)|}{n}& \ll (\log T)^{1/2} G_3\left(1;\frac{1}{2},l_1,m_1,l_{22}\right)\\
& + O \left(\frac{M_{3}}{(\log T)^{1/2}}\right).
\end{align}
We estimate the contribution from the first term of (\ref{fsum}); the contribution of the second term is analogous. We thus study
\begin{align*}
&\sum_{m_1} \frac{d_{1/2}(l_{21}m_1)|\l_f(l_{21}m_1)| G_3\left(1;l_1,m_1,l_{22}\right)}{m_1^s}\\
&= \prod_p \left(\frac{\sum_{k,k'=0}^\infty |d_{1/2}(p^k) \l_f(p^k) \l_f(p^{k'}) d_{1/2}(p^{k+k'}) \l_f(p^{k+k'})|p^{-k'-ks}}{\sum_{k=0}^\infty \l_f(p^k)^2 d_{1/2}(p^k)p^{-k}}\right) G_4(s;l_1,l_2),
\end{align*}
where 
$$G_4(s;l_1,l_2)=\prod_{p| l_1l_2}  G_{4,p}(s;l_1,l_2),$$
and $G_{4,p}(s;l_1,l_2)$ is given by 
\begin{small}
$$ \frac{\displaystyle \sum_{k,k'=0}^\infty |d_{\frac{1}{2}}(p^{\nu_p(l_{21})+k})\l_f(p^{\nu_p(l_{21})+k}) \l_f(p^{\nu_p(l_{22})+k'}) d_{\frac{1}{2}}(p^{\nu_p(l_1)+k+k'}) \l_f(p^{\nu_p(l_1)+k+k'}) |p^{-k'-ks}}{\displaystyle\sum_{k,k'=0}^\infty |d_{\frac{1}{2}}(p^k) \l_f(p^k) \l_f(p^{k'}) d_{\frac{1}{2}}(p^{k+k'}) \l_f(p^{k+k'})|p^{-k'-ks}}.$$
\end{small}
\begin{claim}
Let $s= \sigma +i\tau$; there exists a function, $D$, bounded and holomorphic in the region $\sigma>1-c/\log |\tau|$ such that
\begin{small}
 $$\prod_p \frac{\displaystyle \sum_{k,k'=0}^\infty \frac{|d_{\frac{1}{2}}(p^k) \l_f(p^k) \l_f(p^{k'}) d_{\frac{1}{2}}(p^{k+k'}) \l_f(p^{k+k'})|}{p^{k'+ks}}}{\displaystyle \sum_{k=0}^\infty \l_f(p^k)^2 d_{\frac{1}{2}}(p^k)p^{-k}} = \zeta^{1/4}(s) L^{1/4}(sym^2f,s) D(s).$$
 \end{small}
\end{claim}

\begin{proof}
We have 
\begin{align*}
&\prod_{p} \frac{\sum_{k,k'=0}^\infty d_{1/2}(p^k) \l_f(p^k) \l_f(p^{k+k'}) d_{1/2}(p^{k+k'}) \l_f(p^{k'}) p^{-ks-k'}}{\sum_{k=0}^\infty \l_f(p^k)^2d_{1/2}(p^k) p^{-k}}\\
&= \prod_p \left(1 + d_{1/2}(p)^2 \l_f(p)^2 p^{-s} + O(p^{-s-\frac{1}{2}}) \right),
\end{align*}
and the claim follows immediately.
\end{proof}
We let $M_4(l_1l_2)$ be a positive multiplicative function, supported on squarefree integers, such that
$$M_4(p)^{\pm 1} \leq (1+C_4p^{-\delta_4}),$$
for some constant $C_4>0$ and some $\delta_4 > 0$, chosen so that 
$$\prod_{p|l_1l_2} \sup_{\Re(s)> 1- c/\log |\tau|} G_{4,p}(s;l_1,l_2) \leq |\l_f(l_1l_2)| d_{1/2} (l_{21}l_1) M_4(l_1l_2),$$
and use Lemma \ref{tenen} to conclude the following lemma.
\begin{lemma}\label{flemma}
For $l_1,l_2<T^{\xi-\epsilon}$, we have 
\begin{align*}
&\sum_{nl_1,l_{21}m_1<T^\xi} \frac{|\l_f(l_{22}n) d_{1/2}(l_{21}m_1)d_{1/2}(l_1m_1n)\l_f(l_{21}m_1)\l_f(m_1nl_1)|}{nm_1}\\
&\hspace{5 mm} \ll(\log T)^{3/4}  G_4(1;l_1,l_2) + O\left(M_{4} (\log T)^{-1/4}  \right) 
\end{align*}
\end{lemma}
\subsubsection{Estimating the outer sums} We estimate the contribution from the first term of Lemma \ref{flemma} to II, the second term being treated similarly. We notice that letting $Z=\exp(\log^{2/3} N)$  the contribution from $\max(l_1,l_2) > Z $ is negligible, and thus only care to estimate
\begin{align*}
&\sum_{g<T^{1-3\xi}} r(g)^2\l_f(g)^2 \sum_{\substack{l_1,l_2 < T^{1-3\xi}/g, Z \\ (l_1,l_2)=(l_1l_2,g)=1}} \frac{\l_f(l_1l_2)r(l_1l_2)}{\sqrt{l_1l_2}} \sum_{l_{21}l_{22}=l_2} G_4(1; l_1,l_2)\\
&\ll \sum_{g<T^{1-3\xi}} r(g)^2\l_f(g)^2 \sum_{\substack{l_1,l_2 < T^{1-3\xi}/g, Z \\ (l_1,l_2)=(l_1l_2,g)=1}} \frac{\l_f(l_1l_2)r(l_1l_2)}{\sqrt{l_1l_2}}d_{1/2}(l_1)   d_{3/2}(l_2) \tilde{G}_4(l_1l_2),
\end{align*}
where $\tilde{G}_4$ is a multiplicative function supported on squarefree integers such that
$$\tilde{G}_4(l) = \prod_{p|l} (|\l_f(p^{\nu_p(l)})| + O(p^{-\delta_5}),$$
for some $\delta_5 >0$. By Lemma \ref{r2g2} and \ref{rr2g} we obtain that the contribution of II is bounded by
\begin{align*}
 &(\log T)^{3/4} \prod_p (1+r(p)^2\l_f(p)^2) \prod_p \left(1+ \frac{\l_f(p)^2 r(p)}{2\sqrt{p}}\right) \prod_p\left(1+\frac{3\l_f(p)^2 r(p)}{2\sqrt{p}}\right) I_T\\
& \ll (\log T)^{3/4} \prod_p (1+r(p)^2\l_f(p)^2) \prod_p \left(1+ \frac{2\l_f(p)^2 r(p)}{\sqrt{p}}\right)I_T,
\end{align*}
so that dividing by $NW$ we obtain an acceptable upper bound towards (\ref{signedmoment}).
\subsection{Estimating I}
We have reduced the proof of the upper bound of Proposition \ref{prop1} to bounding I. We will do so by showing it is bounded by II. By the approximate functional equation \cite[p. 98]{IwaniecKowalski},  we have that 
\begin{align*}
L\left(f,\frac{1}{2}+i\nu\right) &= \sum_n \frac{\l_f(n)}{n^{1/2+i\nu}}  V_{\nu}(n) + \Delta(i\nu) \sum_n \frac{\l_f(n)}{n^{1/2-i\nu}} V_{-\nu}(n),
\end{align*}
where 
$$V_{\nu}(y):= \frac{1}{2\pi i} \int_{(3)} y^{-u}e^{u^2} \frac{L_\infty\left(\frac{1}{2}+i\nu+u\right)}{L_\infty\left(\frac{1}{2}+i\nu\right)} \frac{\mathrm{d}u}{u}.$$
We may thus write
\begin{align*}
\int_\R L\left(f,\frac{1}{2}+i\nu\right)\left(\frac{m_2l_2}{m_1l_1}\right)^{i\nu} \Delta'(-i\nu) \frac{\mathrm{d}\nu}{\cosh\left(\frac{T+\nu}{H}\right)}= S_1 + S_2,
\end{align*}
where
$$S_1 := \sum_n \frac{\l_f(n)}{n^{1/2}} \int_\R\left(\frac{m_2l_2}{m_1l_1n}\right)^{i\nu} V_\nu(n) \Delta'(-i\nu) \frac{\mathrm{d}\nu}{\cosh\left(\frac{T+\nu}{H}\right)},$$
and 
$$S_2:= \sum_n \frac{\l_f(n)}{n^{1/2}} \int_\R \left(\frac{m_2l_2n}{m_1l_1}\right)^{i\nu}V_{-\nu}(n) \Delta(i\nu)\Delta'(-i\nu) \frac{\mathrm{d}\nu}{\cosh\left(\frac{T+\nu}{H}\right)}.$$
We note that by the support of $V_{\nu}$ \cite[p. 100]{IwaniecKowalski}, we only need to consider the contribution from $|\nu| \asymp T$ and $n \ll T^{1+\epsilon}$, for both $S_1$ and $S_2$ . We also note that in the definition of $V_\nu$ only the contribution from $u\ll T^\epsilon$ is non-negligible. We thus estimate 
\begin{align*}
S_1& = \sum_{n\ll T^{1+\epsilon}} \frac{\l_f(n)}{2\pi n^{1/2}} \int_{u \ll T^\epsilon} \frac{e^{(3+iu)^2}}{n^{3+iu} (3+iu)} K_T(n;m_1,m_2,l_1,l_2,u)\mathrm{d}u+O(T^{-A}), 
\end{align*}
where $K_T(n;m_1,m_2,l_1,l_2,u)$ is defined to be
$$\int_\R \left(\frac{m_2l_2}{m_1l_1n}\right)^{i\nu}\frac{L_\infty\left(\frac{7}{2}+i(\nu+u)\right)}{L_\infty\left(\frac{1}{2}+i\nu\right)} \Delta'(-i\nu) W(\nu) \frac{\mathrm{d}\nu}{\cosh\left(\frac{T+\nu}{H}\right)},$$
and $W$ is a smooth function supported on $[-2T,-T/2]$ such that $W^{(j)}(x) \ll x^{-j}$ for all $j\geq 0$.

We recall that by (\ref{logder}),
$$\Delta'(-i\nu) = -\frac{\Delta(-i\nu)}{2} \log\left(\frac{\frac{1}{16}+\frac{1}{4}((r+\nu)^2+(\nu-r)^2)+(\nu^2-r^2)^2}{16\pi^4}\right)+O(|\nu|^{-1+\epsilon}),$$ 
and 
\begin{align*}
\frac{L_\infty\left(\frac{7}{2}+i(\nu+u)\right)}{L_\infty\left(\frac{1}{2}+i\nu\right)}\Delta(-i\nu) &= c_1\pi^{-2i\nu}\left|\frac{r+u+\nu}{2e}\right|^{i\frac{r+u+\nu}{2}+\frac{5}{4}} \left|\frac{\nu+u-r}{2e}\right|^{i\frac{\nu+u-r}{2}+\frac{5}{4}}\\
&\times  \left|\frac{r+\nu}{2e}\right|^{i\frac{(r+\nu)}{2}+\frac{1}{4}} \left|\frac{\nu-r}{2e}\right|^{i\frac{(\nu-r)}{2}+\frac{1}{4}} e^{\frac{\pi u}{2}}(1+ O(|\nu|^{-1})),
\end{align*}
for some absolute constant $c_1$. We then write 
$$K_T(n;m_1,m_2,l_1,l_2,u)= \int_\R g_T(\nu)e(f_T(\nu)) \mathrm{d}\nu,$$
where
\begin{align*}
g_T(\nu) &=  c_2 \frac{\log\left(\frac{\frac{1}{16}+\frac{1}{4}((r+\nu)^2+(\nu-r)^2)+(\nu^2-r^2)^2}{16\pi^4}\right)+ O(|\nu|^{-1+\epsilon})}{\cosh\left(\frac{T+\nu}{H}\right)}W(\nu)\\
&\times \left|\frac{r+u+\nu}{2}\right|^{5/4}\left|\frac{\nu+u-r}{2}\right|^{5/4}\left|\frac{r+\nu}{2}\right|^{1/4}\left|\frac{\nu-r}{2}\right|^{1/4}e^{\pi u},
\end{align*}
fore some absolute constant $c_2$, and
\begin{align*}
2\pi f_T(\nu) &= \nu \log\left(\frac{m_2l_2}{nm_1l_1\pi^2}\right) + \frac{r+\nu}{2}\log\left|\frac{r+\nu}{2e}\right|+\frac{\nu-r}{2}\log\left|\frac{\nu-r}{2e}\right|\\
&+ \frac{r+u+\nu}{2}\log\left|\frac{r+u+\nu}{2e}\right| + \frac{\nu+u-r}{2}\log\left|\frac{\nu+u-r}{2e}\right|.
\end{align*}
We now wish to run a stationary phase analysis on $K_T$, and we therefore compute
\begin{align*}
2\pi f'_T(\nu)&= \log\left(\frac{m_2l_2}{nm_1l_1\pi^2}\right) +\frac{1}{2} \log\left|\frac{r+\nu}{2e}\right| +\frac{1}{2}\log\left|\frac{\nu-r}{2e}\right| + 2\\
&+ \frac{1}{2}\log\left|\frac{r+u+\nu}{2e}\right| + \frac{1}{2} \log\left|\frac{\nu+u-r}{2e}\right|.
\end{align*}
We note that in the support of the integral, $|f_T'(\nu)| \geq 1$ as otherwise we would require to have $m_2l_2T^2\asymp nm_1l_1$, however the right hand side is always bounded by $T^{2-\epsilon}$. By repeated integration by parts, we find 
$$K_T(n;m_1,m_2,l_1,l_2,u) \ll e^{\pi u} T^{-A},$$
so that the contribution from $S_1$ is negligible. Similarly, we now study
$$S_2= \sum_{n\ll T^{1+\epsilon}} \frac{\l_f(n)}{2\pi n^{1/2}} \int_{u \ll T^\epsilon} \frac{e^{(3+iu)^2}}{(\pi n)^{3+iu}(3+iu)} \tilde{K}_T(n;m_1,m_2,l_1,l_2,u) \mathrm{d}u +O(T^{-A}),$$
where $\tilde{K}_T(n;m_1,m_2,l_1,l_2,u)$ is defined to be
$$\int_\R \left(\frac{m_2l_2n}{m_1l_1}\right)^{i\nu} \frac{L_\infty\left(\frac{7}{2}+i(u-\nu)\right)}{L_\infty\left(\frac{1}{2}-i\nu\right)} \Delta(i\nu)\Delta'(-i\nu) \frac{W(\nu)\mathrm{d}\nu}{\cosh\left(\frac{T+\nu}{H}\right)}.$$
We write
$$\tilde{K}_T(n;m_1,m_2,l_1,l_2,u) = \int_\R \tilde{g}_T(\nu) e(\tilde{f}_T(\nu)) \mathrm{d}\nu,$$
where up to a constant, $\tilde{g}_T(\nu)$ is given by
\begin{align*}
 &e^{-\frac{\pi u}{2}} \left|\frac{u+r-\nu}{2}\right|^{5/4} \left|\frac{u-\nu-r}{2}\right|^{5/4} \left|\frac{r-\nu}{2}\right|^{1/4} \left|\frac{r+\nu}{2}\right|^{1/4}\\
  &\times \log\left(\frac{\frac{1}{16}+\frac{1}{4}((r+\nu)^2+(\nu-r)^2)+(\nu^2-r^2)^2}{16\pi^4}\right) \frac{W(\nu)}{\cosh\left(\frac{T+\nu}{H}\right)}(1+ O(|\nu|^{-1+\epsilon})),
  \end{align*}
and
\begin{align*}
2\pi \tilde{f}_T(\nu)&= \nu \log\left(\frac{m_2l_2n}{m_1l_1}\right) + \frac{u+r-\nu}{2} \log\left|\frac{u+r-\nu}{2e}\right| + \frac{u-\nu-r}{2} \log\left|\frac{u-\nu-r}{2e}\right|\\
& +\frac{\nu-r}{2} \log\left|\frac{r-\nu}{2e}\right| +\frac{\nu+r}{2}\log\left|\frac{\nu+r}{2e}\right|.
\end{align*}
We compute
$$2\pi \tilde{f}_T'(\nu) = \log\left(\frac{m_2l_2n}{m_1l_1}\right) -\frac{1}{2} \left(\log\left(1+\frac{u}{r-\nu}\right) + \log\left(1-\frac{u}{\nu+r}\right)\right),$$
We thus see that if $m_2l_2n \not= m_1l_1,$ then $\tilde{f}_T'(\nu) \gg T^{\xi-1}$. Computing higher derivatives, one finds that $\tilde{f}_T^{(j)} \ll T^{\epsilon - j}$, so that by \cite[Lemma 8.1]{BKY} one concludes that 
$$\tilde{K}_T(n;m_1,m_2,l_1,l_2,u) \ll O(T^{-A}).$$
Using the bound $V_{-\nu}(n) \ll 1$, we consider thus have 
$$S_2\ll \delta_{m_2l_2n=m_1l_1} \frac{|\l_f(n)|}{n^{1/2}}   I_{T} + O(T^{-A}).$$
The contribution from the main term thereof to I is therefore bounded by \\ \\
\resizebox{12.5 cm}{!}{$\displaystyle \sum_{l_1,l_2 < T^{1-3\xi}} r(l_1)r(l_2)|\l_f(l_1)\l_f(l_2)| \sum_{\substack{m_1,m_2<T^\xi \\ m_2l_2n=m_1l_1}} \frac{d_{\frac{1}{2}}(m_1)d_{\frac{1}{2}}(m_2) |\l_f(m_1)\l_f(m_2)\l_f(n)|}{(nm_1m_2)^{1/2}} I_T,$}\\ \\
which is the same sum as that appearing in II. This concludes the proof of the upper bound of Proposition \ref{prop1}.

\bibliography{refs}
\bibliographystyle{plain} 

\end{document}